\documentclass[oneside,reqno,a4paper,11pt]{amsart}
\usepackage{amsfonts}
\usepackage{amsmath,amsfonts,amssymb,amscd,mathrsfs}
\usepackage{latexsym, amscd, epsfig}
\usepackage{bbm,color,soul,supertabular,longtable,verbatim}
\usepackage[warning, easyold,math]{easyeqn}
\usepackage{graphicx}
\usepackage{color}

\newtheorem {theorem}{Theorem}[section]
\newtheorem {lemma}[theorem]{{\bf Lemma}}
\newtheorem {corollary}[theorem]{{\bf Corollary}}
\newtheorem {prop}[theorem]{{\bf Proposition}}
\theoremstyle{remark}
\newtheorem {remark}{{\bf Remark}}[section]

\theoremstyle{problem}

\theoremstyle{definition}

\theoremstyle{plain} \numberwithin {equation}{section}

\everymath{\displaystyle}

\def\XXint #1#2#3{{\setbox 0=\hbox {$#1{#2#3}{\int }$}
\vcenter {\hbox {$#2#3$}}\kern -.5\wd 0}}

\def\b#1{\overline{#1}}
\def\M{\mathcal{M}} 
\def\mclH{\mathcal{H}}
\def\R{{\mathbb{R}}}
\def\B{{\mathcal{B}}}

\def\BB{\mathfrak{B}}

\def\cH{\check{H}}
\def\cfH{\check{\mathfrak{H}}}
\def\cW{\check{W}}
\def\cmB{\check{\mathcal{B}}}
\def\mfW{\mathfrak{W}}
\def\mfA{\mathfrak{A}}
\def\msC{\mathscr{C}}
\def\mfC{\mathfrak{C}}

\def\G{{\mathcal{G}}}
\def\ba{\begin{array}}
\def\ea{\end{array}}
\def\be{\begin{equation}}
\def\ee{\end{equation}}

\def\bes{\begin{mysubequations}}
\def\ees{\end{mysubequations}}

\def\e{\varepsilon}

\def\g{\nabla}
\def\a{\alpha}
\def\b{\beta}
\def\d{\delta}

\def\O{\Omega}
\def\Div{\text{div}}

\def\f{\frac}
\def\p{\partial}
\def\o{\omega}

\def\tpsi{\tilde{\psi}}

\def\f{\frac}
\def\p{\partial}
\def\o{\omega}

\begin{document}

\vspace{1cm}

\title[Subsonic Flows Past a Wall]{Two Dimensional Subsonic Euler Flows\\
 Past a Wall or a Symmetric Body$^*$}

\author[Chao Chen, \ \ Lili Du, \ \ Chunjing Xie,\ \ Zhouping Xin]{Chao Chen$^{\lowercase{a,1}}$, \ \ Lili Du$^{\lowercase{b,2}}$, \ \ Chunjing Xie$^{\lowercase{c,3}}$\  \ Zhouping Xin$^{\lowercase{d,4}}$}
\thanks{$^*$ Chen is supported by NSFC grant 11301079. Du is supported in part by NSFC grant 11171236, PCSIRT (IRT1273) and Sichuan Youth Science and Technology Foundation 2014JQ0003. Xie is
supported in part by NSFC grant 11241001, Shanghai Chenguang
Program,  Shanghai Pujiang  program 12PJ1405200, and the Program for
Professor of Special Appointment (Eastern Scholar) at Shanghai
Institutions of Higher Learning. Xin is supported in part by Zheng
Ge Ru Foundation, Hong Kong RGC Earmarked Research Grants
CUHK4041/11P, and CUHK4048/13P, a Focus Area Grant from The Chinese
University of Hong Kong, and a CAS-Croucher Joint Grant.}
\thanks{$^1$ E-Mail: chenchao\_math@sina.cn. $^2$ E-Mail: dulili@scu.edu.cn. $^3$ E-mail: cjxie@sjtu.edu.cn. $^4$ E-mail:zpxin@ims.cuhk.edu.hk }
\maketitle
\begin{center}
$^a$ School of Mathematics and Computer Science, Fujian Normal University,

         Fuzhou, Fujian 350108, P. R. China.

$^b$ Department of Mathematics, Sichuan University,

          Chengdu 610064, P. R. China.

           $^c$ Department of Mathematics, Institute of Natural Sciences,

           Ministry of Education Key Laboratory of Scientific and Engineering Computing,

           Shanghai Jiao Tong University, Shanghai, 200240, P. R. China.

$^d$ The Institute of Mathematical Sciences and Department of Mathematics,

The Chinese University of Hong Kong, Shatin, Hong Kong.

\end{center}


\begin{abstract} The existence and uniqueness of two dimensional steady
compressible Euler flows past a wall or a symmetric body are established. More
precisely, given positive convex horizontal veloicty in the upstream, there
exists a critical value $\rho_{cr}$ such that if the incoming density
in the upstream is larger than  $\rho_{cr}$, then
there exists a subsonic flow  past a wall. Furthermore,
$\rho_{cr}$ is critical in the sense that there is no such subsonic
flow if the density of the incoming flow is less than
$\rho_{cr}$. The subsonic flows possess large vorticity and positive
horizontal velocity above the wall except at the corner points on
the boundary. Moreover, the existence and uniqueness of a two
dimensional subsonic  Euler flow past a symmetric body are also
obtained when the incoming velocity field is a general small
perturbation of a constant velocity field and the density of the incoming flow is larger than a critical value. The
asymptotic behavior of the flows is obtained with the aid of some
integral estimates for the velocity field and its far field states.
\end{abstract}

\

\begin{center}
\begin{minipage}{5.5in}

\

\end{minipage}
\end{center}

\


\section{Introduction and Main Results}
One of the most important problems in aerodynamics is to study flows
past a body. Mathematical investigation for this problem has a long
history. When the flow is irrotational, the study on subsonic flows
past a body is quite mature. The existence of two dimensional
subsonic irrotational flows past a smooth body with small free
stream Mach number was obtained by Shiffman \cite{Shiffman}. When
the free stream Mach number is less than a critical number, Bers
\cite{Bers1} proved the existence of two dimensional subsonic
irrotational flows around a general body and also showed that the
maximum of Mach numbers approaches to one as the free stream Mach
number approaches to the critical value. The uniqueness and
asymptotic behavior of subsonic irrotational plane flows were
studied in \cite{Finn572d}. The existence of three dimensional
subsonic irrotational flows around a smooth body were established in
\cite{Finn573d, Dong} when the free stream Mach number is less than
a critical number. The fine properties of two dimensional smooth
subsonic-sonic irrotational flows around a body were investigated in
\cite{Gilbarg54}. The existence of weak solutions for subsonic-sonic
flows by a compensated compactness method was obtained in
\cite{CDSW, HWW}. A significant result by Morawetz shows that, in
general, smooth transonic flows past a profile are unstable with
respect to small perturbations for the profile, see \cite{Morawetz1,
Morawetz2, Morawetz3}. Hence one has to deal with transonic flows
with discontinuities where in general the flows have non-zero
vorticity in the subsonic region.

The vorticity in compressible ideal flows
is important not only mathematically but also physically. The main
purpose of this paper is to investigate the existence and uniqueness
of two dimensional subsonic Euler flows with non-zero vorticity past
a wall or a symmetric body. As stated in \cite[p.12]{Bers},
``Closely related to the flow around a profile is the flow past a
wall". Subsonic Euler flows with non-zero vorticity in a physical
domain were first established in \cite{XX3} where Xie and Xin studied
subsonic Euler flows through an infinitely long smooth nozzle. The major difficulty for the steady Euler system with non-zero vorticity is that the Euler system is a hyperbolic-elliptic coupled system for subsonic flows. In \cite{XX3}, a physical boundary condition for the hyperbolic mode is proposed in the upstream of the flows and the stream function formulation is used to solve the hyperbolic mode \cite{XX3} so that the steady Euler
system is reduced into a single second order equation with memory. This approach was generalized to subsonic
flows with non-zero vorticities in nozzles in various settings, such
as the flows in periodic nozzles and axially symmetric nozzles, the
non-isentropic flows, see \cite{CX1, DD, DD1, CDX, DL} and
references therein.  In particular, the subsonic Euler flows with
stagnation points and large vorticity in nozzles were studied in \cite{DX,
DXX}.

An attempt for the well-posedness theory for subsonic
Euler flows in half plane was made in \cite{CJ} via the stream
function formulation. However, as mentioned
in \cite[Remark 6.1]{CJ}, the absence of stagnation point in the flow region was not obtained in \cite{CJ} so that the author in \cite{CJ} failed to get the equivalence between the stream function formulation and the Euler system, in particular, the uniqueness of the solutions of the original Euler system. Furthermore, the crucial techniques in \cite{CJ} rely on the estimate for elliptic equations which are small perturbations of the Laplace equation in half plane, so even for the reduced problem for the stream function, the existence and uniqueness of the solution for the problem of the stream function were achieved  in \cite{CJ} only when the incoming velocity is a sufficiently smaller perturbation of a small constant state and the incoming density is a large constant.    Our aim in this paper is to prove the existence of subsonic Euler flows, in particular, the flows with large vorticity, as long as the density in the upstream is larger than a critical value. We also prove that subsonic flows above a wall do not have stagnation points and the streamlines of the flows have simple topological structure. The region above the wall can be approximated by a sequence of nozzles so that the analysis in \cite{XX3,DX,DXX} for general quasilinear equation with memory term helps solve these approximated problems. However, the estimates in \cite{XX3,DX,DXX} depend on the height of the nozzles, so one of the key issues in this paper is to prove a series of uniform estimates independent of the nozzle height.


Two-dimensional steady isentropic ideal flows are governed by the
following Euler system \be\label{a0}
\left\{\ba{l} \p_{x_1}(\rho u) +\p_{x_2}(\rho v)=0,\\
\p_{x_1}(\rho u^2)+\p_{x_2}(\rho uv)+\p_{x_1}p=0,\\
\p_{x_1}(\rho uv)+\p_{x_2}(\rho v^2)+\p_{x_2}p=0, \ea\right. \ee
where $(u,v)$ is the velocity, $\rho$ is the density, and $p=p(\rho)$
is the pressure of the flow. In this paper, for the simplicity of presentation,  we consider the
polytropic gas for which the equation of state is $ p=
\rho^{\gamma}$ with the constant $\gamma>1$ called
the adiabatic exponent. The local sound speed and Mach number of the flow are defined to be $$
c(\rho)=\sqrt{p'(\rho)}=\sqrt{\gamma\rho^{\gamma-1}}\,\, \text{and}\,\, M=\f{\sqrt{u^2+v^2}}{c(\rho)}, $$ respectively.
The flow is said to be subsonic if $M<1$, and supersonic if $M>1$.

Consider the flow past a wall $\Gamma=\{(x_1, f(x_1)):x_1\in \mathbb{R}\}$, i.e., we study the solution of \eqref{a0} in  $\O$ defined by
 \be\label{a1}
\Omega= \left\{(x_1,x_2)\in\R^2|\ x_2>f(x_1),\ \ -\infty< x_1 <
+\infty\right\}. \ee
 $f(x_1)$ is assumed to be a nonnegative continuous function
satisfying \be\label{a23}
f(x_1)>0\quad \text{for}\,\, x_1\in (0,1)\quad \text{and}\quad  f(x_1)\equiv 0 \quad\text{for}\quad
x_1\in (-\infty, 0]\cup [1, \infty).\ee Furthermore, the curve $\{(x_1,f(x_1)): x_1\in [0, 1]\}$ is a $C^{1, \alpha}$ smooth curve.
\begin{figure}[!h]
\includegraphics[width=150mm]{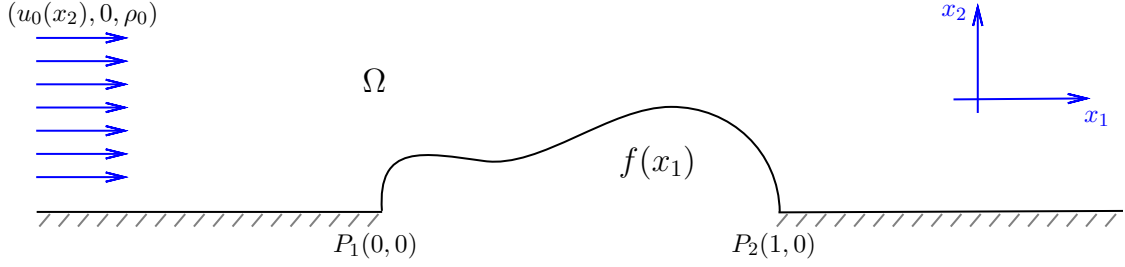}
\caption{Subsonic flows past a wall}\label{f1}
\end{figure}

The solid wall $\Gamma$ is assumed to be impermeable and thus,
\be\label{a01} (u,v)\cdot \vec n=0 \ \ \ \text{on}\ \ \ \Gamma,\ee
where $\vec n$ is the unit outward normal to the boundary $\Gamma$.
Furthermore, the flow velocity is prescribed in the upstream as
 \be\label{a02}
(u(x_1,x_2),v(x_1,x_2))\rightarrow (u_0(x_2), 0)\ \ \ \text{as}\ \ \
x_1\rightarrow-\infty.\ee
Finally, the density in the upstream is given as follows
\begin{equation}\label{upMach}
\rho(x_1,x_2)\rightarrow \rho_0\ \ \ \text{as}\ \ \
x_1\rightarrow-\infty
\end{equation}
where $\rho_0$ is a constant. Furthermore, if $f$ is not differentiable at $x_1=0$ and $1$, then  the flow is required to satisfy the
Kutta-Joukowski condition (cf. \cite{Bers}) at the corner points
$\{P_1, P_2\}$ where $P_1=(0,0)$ and $P_2=(1,0)$, i.e., the flow velocity is continuous at $\{P_1,
P_2\}$.

 Let us state the main results as follows.
\begin{theorem}\label{th1}  Suppose that  the  upstream horizontal velocity $u_0(x_2)$ in
\eqref{a02} satisfies that \be\label{a2} u_0(x_2)\in C^2(\R^+),\ \ \
u_0(x_2) >0, \ \ u''_0(x_2)\geq 0, \ \ u_0'(0)\leq 0,\ \ \
\lim_{x_2\rightarrow+\infty}u'_0(x_2)=0, \ee and there exists a
$\bar u>0$ such that  \be\label{a3}
\lim_{x_2\rightarrow+\infty}u_0(x_2)=\bar u, \ee then there exists a
critical value $\rho_{cr}>0$, such that if the incoming density $\rho_0$ in \eqref{upMach}
is larger than $\rho_{cr}$, then there exists a
uniformly subsonic flow $(\rho,u,v)\in \left(C^{1,\a}(\O)\cap
C^\beta(\bar\O)\right)^3$ for some $\beta\in (0, \alpha)$, which satisfies the Euler system
\eqref{a0}, the boundary conditions \eqref{a01}-\eqref{upMach}, and
Kutta-Joukowski condition at the corner points $\{P_1, P_2\}$.
Moreover,
\begin{enumerate}
\item[(1)] the flow is uniformly subsonic \be\label{a4} \sup_{\bar
\O}(u^2+v^2-c^2(\rho))<0, \ee and the horizontal velocity is
positive except at the corners \be\label{a5} u>0\ \ \ \ \text{in}\ \
\ \bar\O\backslash\{P_1\cup P_2\}\ \ \ \ \text{ and}\ \ u=0\ \
\text{ at}\ \ P_1\cup P_2;\ee

\item[(2)] the flow satisfies \be\label{a6}
\|(\rho u-\rho_0u_0(x_2),\rho v)\|_{L^2(\O)}\leq C \ee for some
$C>0$ and has the following asymptotic behavior in far fields,
\be\label{a7} \rho\rightarrow\rho_0,\quad\quad (u,v)\rightarrow
(u_0(x_2),0), \ee \be\label{a8} \g\rho\rightarrow0,\quad\quad \g
u\rightarrow (0,u_0'(x_2)),\quad\quad \g v\rightarrow0,
 \ee as $|x_1|\rightarrow \infty$ uniformly for $x_2\in
K_1\Subset(0,+\infty)$,  and \be\label{a16} u\rightarrow\bar u,\ \ \
v\rightarrow 0,\ \ \ \rho\rightarrow\rho_0, \quad \text{as}
\quad x_2\rightarrow+\infty;\ee

\item[(3)] the subsonic flow satisfying Euler
system \eqref{a0}, boundary conditions \eqref{a01}-\eqref{upMach}, \eqref{a5},  and
asymptotic behavior \eqref{a6}-\eqref{a8} is unique.

\item[(4)] $\rho_{cr}$ is the critical incoming density for the existence of
subsonic flow past a wall in the following sense: either
\be\label{a9} \sup_{\bar\O}\f{|(u,v)|}{c(\rho)}\rightarrow1\ \ \ \
\text{as}\ \ \rho_0\rightarrow \rho_{cr}, \ee or there is no $\sigma >0$,
such that for all  $\rho_0\in(\rho_{cr}-\sigma,\rho_{cr})$ there are Euler
flows satisfying \eqref{a0}, subsonic condition \eqref{a4} and
asymptotic behavior \eqref{a6}-\eqref{a8}, and
$$
\sup_{\rho_0\in(\rho_{cr}-\sigma,
\rho_{cr})}\sup_{\bar\O}\f{|(u,v)|}{c(\rho)}<1.
$$
\end{enumerate}
\end{theorem}

If the convexity assumption on $u_0$ in \eqref{a2} is removed, then
the following theorem holds.
\begin{theorem}\label{th2} Suppose that the
horizontal velocity $u_0(x_2)$ in the upstream satisfies that
\be\label{a04} u_0(x_2)\in C^2(\R^+),\ \ \ u_0(x_2)
>0, \  \ \ u_0'(0)\leq0,\ \ \
\lim_{x_2\to \infty}u_0(x_2)=\bar u, \ee for some constant $\bar
u>0$.  Furthermore, there exist constants $k>1$ and $\e > 0$ such
that \be\label{a05}
 \left|\frac{d^i u_0(x_2)}{dx_2^{i}} \right|\leq\f\e{(1+x_2)^{k+i}}\ \ \ \text{for } i=1, 2\ \
\text{ and}\ \ x_2>0.\ \ee There exist an $\e_0>0$  and a critical
value $\rho_{cr}>0$ such that if $\e$ in \eqref{a05} satisfies
$\e\in (0, \e_0)$, and the incoming density $\rho_0$ in the upstream
is larger than $\rho_{cr}$, then the problem \eqref{a0},
\eqref{a01}-\eqref{upMach} admits
 a uniformly subsonic flow $(\rho,u,v)\in
\left(C^{1,\a}(\O)\cap C^\beta(\bar\O)\right)^3$ for some $\beta\in (0, \alpha)$ satisfying Kutta-Joukowski condition at the corner points $\{P_1, P_2\}$. Moreover, the
properties (1)-(4) in Theorem \ref{th1} hold.
\end{theorem}

As a direct consequence of Theorem \ref{th2}, one gets the existence
and uniqueness of smooth subsonic ideal flows past a symmetric
obstacle.

\begin{corollary}\label{cor}(Subsonic flow past a symmetric obstacle)
Let  $\tilde\O$ be defined as
$$
\tilde \O=\left\{(x_1,x_2)\in\left.\R^2\right| |x_2| > f(x_1)\right\}.
$$
Suppose that the
horizontal velocity $u_0(x_2)$ in the upstream satisfies
\be\label{a06} u_0(x_2)\in C^2(\R),\ \ \ u_0(x_2)= u_0(-x_2),\ \ \
u_0(x_2)
>0, \ \ \
\lim_{x_2\rightarrow\pm\infty}u_0(x_2)=\bar u>0, \ee and
\eqref{a05}, then there exists a critical value $\rho_{cr}>0$, such
that if the incoming density $\rho_0$ in the upstream is larger than
$\rho_{cr}$, then there exists a uniformly subsonic symmetric
flow $(u,v,\rho)\in \left(C^{1,\a}(\tilde\O)\cap
C^\beta(\bar{\tilde\O})\right)^3$, which satisfies the Euler system
\eqref{a0}, the boundary conditions \eqref{a01}-\eqref{upMach}.
Moreover, the similar properties for $(\rho,u,v)$ as that in (1)-(4)
in Theorem \ref{th1} hold.
\end{corollary}

\begin{figure}[!h]
\includegraphics[width=130mm]{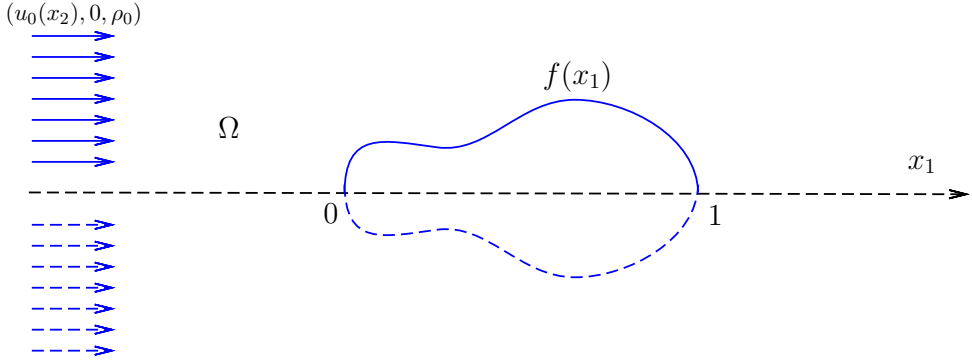}
\caption{Subsonic flows past a symmetric obstacle}\label{f2}
\end{figure}

As $\rho_0$ approaches to $\rho_{cr}$ in Theorems \ref{th1} and \ref{th2}
and Corollary \ref{cor}, combining with the compensated compactness
framework obtained in \cite{CHW}, we have the following theorem.
\begin{theorem}\label{thmlimit}
Assume that $\{(\rho_n,u_n,v_n)\}$ is a sequence of  subsonic Euler flows in $\Omega$ or $\tilde\Omega$ which satisfy \eqref{a01}-\eqref{upMach} with $\rho_0$ replaced by $\rho_0^{(n)}$. Suppose that $\rho_0^{(n)}\downarrow \rho_{cr}$ as $n\to \infty$, then there exists a limiting flow $(\rho, u, v)$ satisfying
\begin{equation}
(\rho_n,u_n,v_n)\rightarrow (\rho,u,v)\,\ \text{a.e. in}\,\ \O\,\ \text{as}\,\ n\rightarrow +\infty.
\end{equation}
Furthermore, $(\rho, u, v)$ satisfies the Euler system \eqref{a0} in the sense of distribution and the boundary condition \eqref{a01} in the sense of normal trace.
\end{theorem}

There are several remarks in order for the main results.

\begin{remark}
The most significant difference between the Euler flows with non-zero vorticity and irrotational flows is that the governing system for subsonic flows with non-zero vorticity is a hyperbolic-elliptic coupled system. In order to solve the hyperbolic mode, one has to prescribe suitable physical boundary conditions for the corresponding hyperbolic mode. The far field boundary condition \eqref{upMach} can be regarded as the boundary condition for the associated hyperbolic mode.
The general form for the boundary conditions in the upstream can be written as
\begin{equation}
(\rho(x_1, x_2), u(x_1, x_2), v(x_1, x_2))\to (\rho_0(x_2), u_0(x_2), v_0(x_2))
\quad \text{as}\quad x_1\to -\infty.
\end{equation}
Suppose that the flows satisfy the far field conditions with high order compatibility conditions, i.e.,
\begin{equation}
\nabla (\rho(x_1, x_2), u(x_1, x_2), v(x_1, x_2))\to \nabla (\rho_0(x_2), u_0(x_2), v_0(x_2))
\quad \text{as}\quad x_1\to -\infty.
\end{equation}
It follows from the Euler system \eqref{a0} that  one has
\begin{equation}\label{eqBC1}
(\rho_0 v_0)_{x_2}=0,\quad (\rho_0 u_0 v_0)_{x_2}=0,\quad (\rho_0 v_0^2+p(\rho_0))_{x_2}=0.
\end{equation}
The first equation in \eqref{eqBC1} shows that
\begin{equation}
\rho_0 v_0\equiv C_1
\end{equation}
for some constant $C_1$.
The slip boundary condition \eqref{a01} yields that $\rho_0(0) v_0(0)=0$. Hence $C_1=0$. Since we are looking for steady subsonic solutions, $\rho_0\neq 0$. Thus $v_0\equiv 0$. It follows from the third equation in \eqref{eqBC1} that $(p(\rho_0))_{x_2}=0$. This yields that $\rho_0$ must be a constant.
Therefore, the boundary conditions \eqref{a02} and \eqref{upMach} are indeed the  general form for the boundary conditions for subsonic flows past a wall or symmetric body if the solutions satisfy have high order consistency with the boundary conditions.
\end{remark}

\begin{remark}  For subsonic flows with non-zero vorticity, the boundary condition \eqref{upMach} is used to solve the hyperbolic mode in the steady Euler system. When the far field velocity in \eqref{a02} is a constant field, the flows obtained in this paper are subsonic or subsonic-sonic irrotational flows. In fact,  the boundary condition \eqref{upMach} together with \eqref{a02} can be used to determine Bernoulli constant for  subsonic irrotational flows past a body when the far field velocity  in \eqref{a02} is a constant.  In this case, the boundary conditions \eqref{a01}-\eqref{upMach} and
Kutta-Joukowski condition at the corner points $\{P_1, P_2\}$ reduces to the boundary conditions for subsonic irrotational flows
past a body \cite{Bers1}.  Therefore, the boundary conditions we
prescribed for subsonic Euler flows past a wall  not only are suitable physical boundary conditions for the well-posedness theory for subsonic flows with non-zero vorticity past a wall, but also can be regarded as generalizations for the boundary conditions for irrotational flows past a wall,
\end{remark}

\begin{remark}
When the streamlines of the flows have simple topological structure, we  use the stream function formulation to solve the hyperbolic mode in the Euler system so that the hyperbolic-elliptic coupled Euler system is reduced to a single second order quasilinear equation with memory. Therefore, our first key task is to show that the flows past a wall indeed have simple topological structure. Second, even for the problem for the stream function, the memory term makes the equation more complicated than the one for irrotational flows. The governing equation for irrotational flows is a homogenous second order quasilinear equation for which the maximum principle can be applied easily \cite{Finn572d}. Furthermore, the far field of the irrotational flows past a body  is a uniform constant state so that the far fields can be essentially regarded as a single point and the problem can be transformed into a bounded domain problem
\cite{Bers1,Finn572d, Finn573d, Dong}. The memory term for the flows with non-zero vorticity yields non-uniform far field states. Therefore, the far fields for the flows with non-zero vorticity cannot be regarded as a single point so that the problem for subsonic flows with non-zero vorticity past a wall is indeed a problem on a unbounded domain.
\end{remark}

\begin{remark}
We use a sequence of nozzles to approximate the region above the wall so that the ideas in \cite{XX3,DX,DXX} can be used to get approximated solutions. However, the estimates developed in \cite{XX3,DX,DXX} depend on the height of the nozzles. In order to get the existence for flows past a wall,  we have to establish
some uniform estimates for the flows in the approximated nozzles independent of nozzle height. This is also the crucial step to obtain the existence of subsonic Euler flows past a wall.
\end{remark}

\begin{remark}
The uniqueness of solution
in Theorems \ref{th1} and \ref{th2} and Corollary \ref{cor} is
obtained in the class of subsonic Euler flows in which we proved the
existence, i.e. the class of flows which also have positive
horizontal velocity and satisfy the asymptotic behavior
\eqref{a6}-\eqref{a16}. In the future, we hope to prove the
uniqueness of subsonic Euler flows which satisfy only the  the
boundary conditions \eqref{a01}-\eqref{upMach} and Kutta-Joukowski
condition as in the case of irrotational flows in \cite{Finn572d}.
\end{remark}

\begin{remark}
The flows obtained in Theorem \ref{th1} may have large vorticity so
that they may not be close to potential flows, which cannot be
treated as small perturbations of potential flows and the effect of
vorticity in the whole domain has to be analyzed.
\end{remark}

\begin{remark}
Since the flows have stagnation points at those corner points, it is not easy to claim the absence of stagnation point inside the domain. In order to solve the hyperbolic mode in the Euler system, we need to show the absence of stagnation points inside the flow region and get precise far field behavior of the flows so that
the hyperbolic-elliptic coupled Euler system can be reduced to  a second order
quasilinear equation with memory term for the stream function.
Although the flows obtained in Theorem \ref{th2} have small
vorticity, the memory term appeared in the equation for the stream function makes it hard to prove far field behavior of the flows and the absence of stagnation point above the wall.
\end{remark}

\begin{remark}
As mentioned
in \cite[Remark 6.1]{CJ}, the absence of stagnation point in the flow region was not obtained in \cite{CJ} so that the equivalence between the stream function formulation and the Euler system, in particular, the existence and uniqueness of the solutions of the original Euler system are not established in \cite{CJ};
we prove the existence and uniqueness of subsonic flows for the original Euler system when the incoming density is greater than a critical value as long as the horizontal velocity is positive convex or a small perturbation of a constant state, and we also show that the flows do not have stagnation point above the wall.  Even for the problem for the stream function (cf. the problem \eqref{b01}), the existence and uniqueness of the solutions were obtained in \cite{CJ} only  when the incoming velocity is a sufficiently smaller perturbation of a small constant state and the incoming density is a large constant; while we can deal with the flows with large vorticity and the density which is greater than a critical value.   Finally, the crucial estimates in \cite{CJ} are for elliptic equations which are small perturbations of two dimensional Laplace equation so that the author in \cite{CJ} can only deal with the flows with small velocity and vorticity; in order to deal with subsonic flows whose Mach number might approach to one, we have to rebuild the estimate for general quasilinear equation for the stream functions.
\end{remark}

\begin{remark}
Note that we only obtained the weak convergence in Theorem \ref{thmlimit}. We hope to study the regularity and fine structure of these limiting solutions later on. Continuous transonic irrotational flows in nozzles were constructed in \cite{WX1,WX2,WX3} recently. Furthermore, there are a lot of studies on stability of transonic shock solutions in nozzles recently, see \cite{Chen,CSX,LXY,LXY1, XY1, XY} and references therein. For the problems on the stability of transonic shocks, different boundary conditions for the flows at the downstream are prescribed, where the key issue is to study the free boundary problem for subsonic flows which  are usually small perturbations of some given background solutions.
\end{remark}

\begin{remark}
Theorems \ref{th1} and \ref{th2} can be extended easily to subsonic
Euler flows past several smooth bumps (See Figure \ref{f8}).
\end{remark}

\begin{figure}[!h]
\includegraphics[width=130mm]{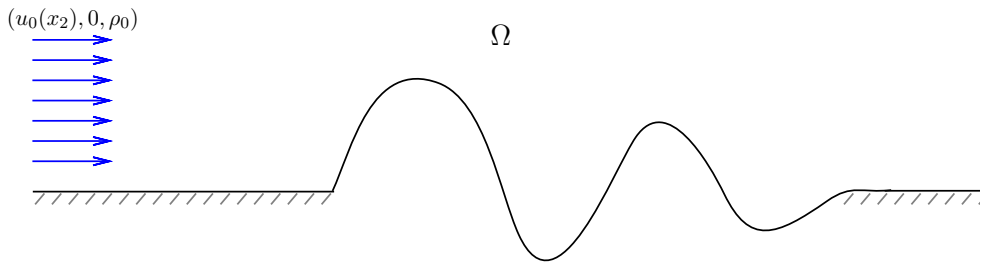}
\caption{Subsonic flows past several bumps}\label{f8}
\end{figure}

\begin{remark}
Corollary \ref{cor} also holds for the subsonic flows past a
symmetric body with a cusp (See Figure \ref{f3}). The only
difference is the horizontal velocity is not necessary zero at the
trailing edge.
\end{remark}

\begin{figure}[!h]
\includegraphics[height=4.5cm, width=130mm]{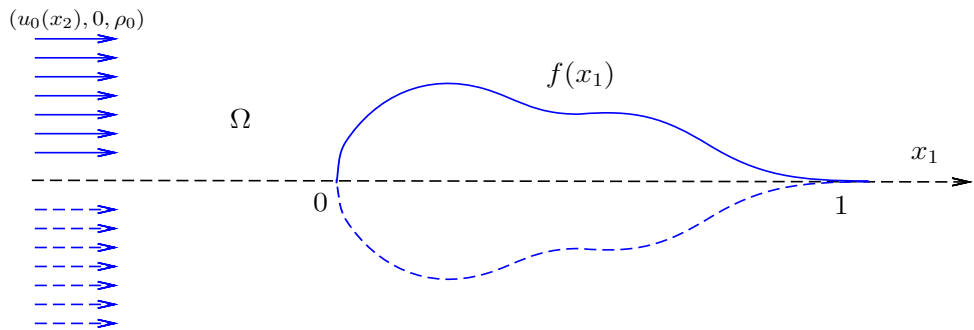}
\caption{Subsonic flows past a symmetric obstacle with
cusp}\label{f3}
\end{figure}

\begin{remark}
All results in this paper also hold for the non-isentropic Euler
equations, provided that we impose the entropy of the flows in the
upstream.
\end{remark}

Before giving the detailed proof for these main results, we present
the main ideas of the proof as follows and the recent progress on
some problems related to flows past a body. Adapting the stream
function formulation developed in \cite{XX3}, we reduce the Euler
system into a single second order elliptic equation with memory for
the stream function. In order to deal with the associated problem
for the stream function in the region above the wall, we approximate
the domain by a sequence of infinitely long nozzles $\Omega_L$
bounded by $\Gamma$ and $\Gamma_L=\{(x_1, x_2)|x_2=L\}$. Combining
the ideas and techniques in \cite{XX3, DX, DXX} gives the existence
of subsonic Euler flows $(\rho_L, u_L, v_L)$ in $\Omega_L$. One of
the key issues in this paper is to obtain some uniform estimates for
the solutions $(\rho_L, u_L, v_L)$ independent of $L$. The uniform
estimates consist of mainly two parts. The first one concerns the
$L^\infty$ estimate for the difference of the stream function and
the corresponding one in the upstream, which is based on a
maximum principle. The second one is the uniform $L^2$
gradient estimate for the same quantity, which follows from a delicate
energy estimate and is used to prove the asymptotic behavior and the
uniqueness of the flows.

The key part of the proof is on the uniform estimate  for the steady
Euler flows in infinitely long nozzles. In fact, the steady Euler
flows in nozzles were studied extensively recently. In his famous
survey \cite{Bers}, Bers asserted that there exists a unique
subsonic irrotational flow in a two-dimensional infinitely long
nozzle, provided the incoming mass flux is sufficiently small. The
rigorous mathematical proof to this assertion was obtained in
\cite{XX1} for flows in 2D nozzles, in \cite{XX2} for flows in
axially symmetric nozzles, and \cite{DXY} for flows in arbitrary
multi-dimensional nozzles, respectively. Furthermore, the existence
of subsonic irrotational flows was established in \cite{DXY, XX1,
XX2} as long as the mass flux is less than a critical value. When
the mass flux approaches the critical value, the associated flows
converge weakly to subsonic-sonic flows, see \cite{HWW, XX1, XX2}.
As we mentioned before, subsonic flows with non-zero vorticity were
first studied in \cite{XX3} via a stream function formulation, which
was generalized in \cite{CX1, DD, DD1, CDX, DL, DX, DXX} and
reference therein. Subsonic Euler flows in finite nozzles, in
particular, three dimensional nozzles were investigated recently in
\cite{CX2,DWX,Weng,W2}.

The rest of this paper is organized as follows. We first adapt the
stream function formulation for the steady Euler flows developed in
\cite{XX3} to reduce the steady Euler system into a single second
order equation in Section \ref{secform}. To obtain the existence of
the subsonic solution to the equation for stream function above a
wall, we construct a series of uniformly subsonic solutions in
infinitely long nozzles as the approximated solutions. Some
uniform estimates for the approximated solutions are obtained in Section
\ref{secexistence} so that the existence of subsonic solution above a wall is established when the density of the incoming
flows is sufficiently large. In order to show that the solution
induced by the stream function solves the compressible Euler system,
in Section \ref{secfine} we prove the asymptotic behavior of the
flows in the far fields and show that the horizontal velocity of the
flows is always positive except at the corners on the boundary. The
uniqueness of a subsonic flow past a wall is given  in Section
\ref{secunique} by the energy method with the aid of $L^2$ integral
estimates. In Section \ref{seccritical},  the existence of
the critical value of the density in the upstream is established so that
the proof of Theorem  \ref{th1} is completed. The Section
\ref{secgeneral} sketches the proof for the existence and the
uniqueness of the compressible subsonic Euler flow past a wall when
the upstream horizontal velocity is a general small perturbation of
a constant, which gives the proof of Theorem  \ref{th2}.
Theorem \ref{thmlimit} is proved in Section \ref{seclimit} with the help of the compensated
compactness framework established in \cite{CHW}. A weighted Poincar\'{e} inequality used in Section \ref{secgeneral}  and its proof are presented in Appendix \ref{secappend}.

\section{Reformulation of the Problem}\label{secform}
It follows from the Euler system \eqref{a0} that one has
\be\label{b39}
(u,v)\cdot\g \left(\f\o\rho\right)=0
\ee
and
\be\label{b44}
(u,v)\cdot\g B=0,
\ee
where
\[
\o= \p_{x_1}v-\p_{x_2}u\quad \text{and}\quad B=\frac{u^2+v^2}{2} + \frac{\gamma
\rho^{\gamma-1}}{\gamma-1}
\]
are the vorticity and the Bernoulli function of the flows,
respectively. The equation \eqref{b44} is nothing but the
Bernoulli's law. Later on, we denote $h(\rho)=\frac{\gamma
\rho^{\gamma-1}}{\gamma-1}$ which is called the enthalpy of the polytropic
gas.

\begin{prop}\label{prop21}
Suppose that
\begin{equation}\label{eq21}
\rho>0\,\,\text{and}\,\, u>0\,\, \text{in}\,\,\Omega,\,\, \text{and}\,\, u\geq\delta\,\, \text{for}\,\,|x|\geq R
\end{equation}
for some positive constants $\delta$ and $R$. Furthermore, assume that
\begin{equation}\label{eq22}
 u,\ \rho,\ \text{and}\ v_{x_2}\ \text{are}\ \text{bounded}, \ \
\text{while}\ v,\ v_{x_1},\ \text{and}\ \rho_{x_2}\rightarrow 0,\ \
\text{as}\ x_1\rightarrow\pm\infty.
\end{equation}
For a smooth flow in $\Omega$ satisfying \eqref{eq21} and \eqref{eq22}, the Euler system
\eqref{a1} is equivalent to the continuity equation, \eqref{b39} and
\eqref{b44}.
\end{prop}
\begin{proof}
A similar result for flows in infinitely long nozzles has been proved in
\cite[Proposition 1]{XX3}. We sketch the proof for the flows past a wall as follows.

Note that the equations \eqref{b39} and \eqref{b44} can be
concluded from the Euler system \eqref{a0}. Hence, it suffices to
prove the validity of the Euler system \eqref{a0} from the
continuity equation and \eqref{b39}-\eqref{b44}. It follows from the continuity equation and \eqref{b39} that
\[
\partial_{x_2}(uu_{x_1}+vu_{x_2}+p_{x_1}/\rho)-\partial_{x_1}(uv_{x_1}+vv_{x_2}+p_{x_2}/\rho)=0.
\]
Therefore,
there exists a function $\Phi$ such that
\begin{equation}\label{Phi}
\partial_{x_1}\Phi=uu_{x_1}+vu_{x_2}+p_{x_1}/\rho\,\ \text{and}\,\ \partial_{x_2}\Phi=uv_{x_1}+vv_{x_2}+p_{x_2}/\rho.
\end{equation}
Furthermore, the straightforward computations show that
\begin{equation}
(u,v)\cdot\nabla\Phi=(u,v)\cdot\nabla B.
\end{equation}
Thus $(u, v)\cdot \nabla \Phi=0$.
Since the horizontal velocity is positive above the wall, each streamline defined by the equation
\begin{equation}\label{strl}
\left\{
\begin{array}{lll}
\frac{dx_1}{ds}=u(x_1(s),x_2(s)),\\
\frac{dx_2}{ds}=v(x_1(s),x_2(s))
\end{array}
\right.
\end{equation}
can be stretched to a unique position in the upstream. To see
this, we need only to claim that any streamline through a point
in $\Omega$ cannot touch the wall or go to the infinity in the
$x_2$ direction. In view of the argument in \cite[Proposition
1]{XX3}, such a streamline is away from the wall. Thus, we need only to show that the streamline cannot go to infinity in the $x_2$-direction with finite $x_1$. In fact, the
assumption \eqref{eq21} that $u$ is uniformly positive for large $x_2$ implies
that $\left|\f{v}{u}\right|$ is uniformly bounded. So, the solution
of the ODE
\[
\frac{dx_2}{dx_1} =\frac{v}{u}(x_1, x_2)
\]
cannot blow up at finite $x_1$.  It follows from \eqref{eq22} that $\Phi$ is a constant as $x_1\to -\infty$. Since $\Phi$ is a constant along each streamline, $\Phi$ is a constant in $\Omega$. Using \eqref{Phi} yields
\[
uu_{x_1}+vu_{x_2}+p_{x_1}/\rho=0\,\ \text{and}\,\ uv_{x_1}+vv_{x_2}+p_{x_2}/\rho=0.
\]
 These, together with the continuity equation give the Euler system \eqref{a1}.
\end{proof}

It follows from the continuity equation that there exists a stream function $\psi$ satisfying
\begin{equation}
\psi_{x_2}=\rho u\quad \text{and}\quad \psi_{x_1}=-\rho v.
\end{equation}
This, together with the slip boundary condition \eqref{a01}, shows that $\psi$ is a
constant on the boundary $\Gamma$. Without loss of
generality, we assume that $\psi=0$ on the boundary
$\Gamma$.

If the flow satisfies \eqref{eq21}, then the streamlines possess a simple topological
structure in the whole domain, so that we can parameterize the
streamlines in the domain by their positions in the upstream.
By the definition of the stream function, we have the following
parametrization for the stream function in the upstream (See Figure
\ref{f14})
\begin{equation}\label{B1}
\psi=\int_{0}^{\kappa(\psi;\rho_0)}\rho_0u_0(s)ds.
\end{equation}
Denote
\begin{equation}\label{Fdef}
F(\psi;\rho_0)=u_0(\kappa(\psi;\rho_0)). \end{equation}
The Bernoulli's law yields \be\label{b02}
\f{|\nabla\psi|^2}{2\rho^2}+h(\rho)
=h(\rho_0)+\f{1}{2}F^2(\psi;\rho_0) = \mathcal{B}(\psi;\rho_0). \ee
\begin{figure}[!h]
\includegraphics[width=130mm]{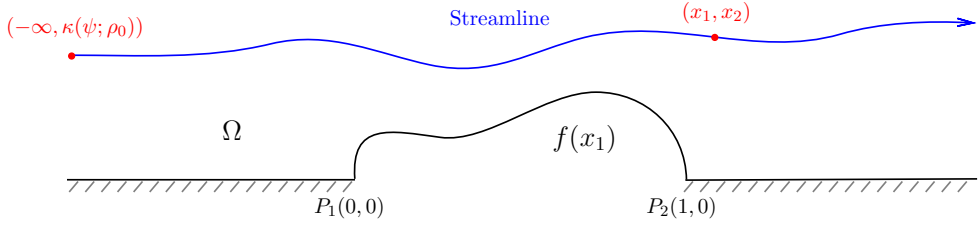}
\caption{Parameterize the streamlines}\label{f14}
\end{figure}

For given $\mathfrak{s}$, there exist unique $\underline{\varrho}(\mathfrak{s})$
and $\bar\varrho(\mathfrak{s})$ such that
\be\label{b03} \f12
\gamma \underline\varrho^{\gamma-1}(\mathfrak{s})+h(\underline\varrho(\mathfrak{s}))=\mathfrak{s}\quad \text{and}\quad h(\bar\varrho(\mathfrak{s}))=\mathfrak{s}. \ee
Define
\begin{equation}\label{defSigma}
\Sigma(\mathfrak{s})=\sqrt{\gamma}\underline\varrho^{\frac{\gamma+1}{2}}(\mathfrak{s})
\end{equation}
which is the critical momentum associated with the sonic state. It is easy to check
from
\begin{equation}\label{eq2Ber}
\frac{\mathfrak{m}^2}{2\rho^2}+h(\rho) =\mathfrak{s}
\end{equation}
that for given $\mathfrak{s}$, the
momentum $\mathfrak{m}$ is a strictly decreasing function of $\rho$ for $\rho \in (\underline\varrho(\mathfrak{s}), \bar\varrho(\mathfrak{s}))$. Therefore, for given $\mathfrak{s}$ and $\mathfrak{m}$, one can find a unique $\rho\in (\underline\varrho(\mathfrak{s}), \bar\varrho(\mathfrak{s}))$ satisfying \eqref{eq2Ber} and denote it by
\begin{equation}
\rho=\mclH\left(\mathfrak{m}^2,\mathfrak{s}\right).
\end{equation}
Later on, we also write
\[
\rho= \mathcal{H}(|\nabla \psi|^2, \B(\psi;\rho_0)) =H\left(|\nabla\psi|^2,\psi; \rho_0\right)
 \]
to emphasize the dependence on the parameter $\rho_0$.
\begin{figure}[!h]
\includegraphics[width=70mm]{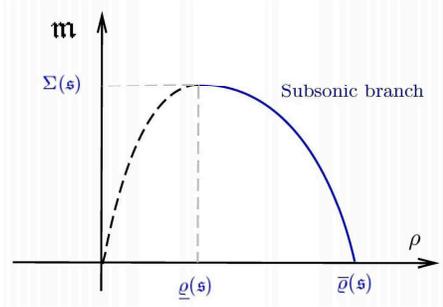}
\caption{The relationship between the momentum and the
density}\label{f5}
\end{figure}

It follows from the vorticity equation \eqref{b39} that the scaled
vorticity $\frac{\omega}{\rho}$ is conserved along each streamline.
Thus, it holds that
$$\o =
- \f{H\left(|\nabla\psi|^2,\psi;\rho_0\right)u_0'(\kappa(\psi;\rho_0))}{\rho_0}.$$
This, together with the identity
$\o=-\Div\left(\f{\g\psi}{\rho}\right)$, gives the equation for the
stream function $\psi$ as follows
\be\label{a10}
\text{div}\left(\f{\nabla\psi}{H\left(|\nabla\psi|^2,\psi;\rho_0\right)}\right)=u_0'(\kappa(\psi;\rho_0))H(|\nabla\psi|^2,\psi;\rho_0)/\rho_0.
\ee
The straightforward computations for \eqref{B1} and \eqref{Fdef} yield
\[
u_0'(\kappa(\psi;\rho_0))=\rho_0F(\psi;\rho_0)F'(\psi;\rho_0),
\]
where $F'$ is the derivative with respect to $\psi$ of $F$.
 Therefore, the problem for the Euler system \eqref{a0} with boundary condition \eqref{a01} is equivalent to
\be\label{b01}\left\{\ba{l}
\text{div}\left(\f{\nabla\psi}{H(|\nabla\psi|^2,\psi;\rho_0)}\right)=F(\psi;\rho_0)F'(\psi;\rho_0) H(|\nabla\psi|^2,\psi;\rho_0)\quad\text{in}\quad\O,\\
\psi=0\ \ \ \text{on}\ \ \ \p\O.\ea\right. \ee
Note that the equation in \eqref{b01} can be written as \be\label{a11}
\left(\left(1-\frac{|\nabla\psi|^2}{\gamma H^{\gamma+1}}\right)\delta_{ij} +\frac{\partial_i\psi\partial_j\psi}{\gamma H^{\gamma+1}}
\right)\partial_{ij}\psi ={F(\psi;\rho_0)F'(\psi;\rho_0){H}^2}.
\ee
It is easy to see that the coefficient matrix of the equation \eqref{a11} has two eigenvalues $\lambda=1-\frac{|\nabla\psi|^2}{\gamma H^{\gamma+1}}$ and $\Lambda=1$. As $|\nabla \psi|^2\to \gamma H^{\gamma+1}$, i.e., the flows go to sonic, the equation \eqref{a11} becomes degenerate
elliptic. This is one of the major
difficulty for solving the problem \eqref{b01}. Another main
difficulty here is that the domain $\Omega$ is unbounded in every
direction.

\section{Existence of subsonic solution  with  large incoming density}\label{secexistence}
In this section, we  establish the existence of subsonic
solutions for the boundary value problem \eqref{b01} when the incoming density $\rho_0$ is sufficiently large.  Since the
upstream is uniformly subsonic, the uniform density $\rho_0$ in the upstream indeed
has a lower-bound $\rho_0^*=\left(\frac{1}{\gamma}\sup
u_0^2(x_2)\right)^{\f1\gamma}$. In Sections 3-6, we assume that the hypotheses in Theorem \ref{th1} hold. First, we have the following proposition.
\begin{prop}\label{prop31}
There exists a $\bar\rho_0\in (\rho_0^*, \infty)$ such that if
$\rho_0>\bar \rho_0$,  the problem \eqref{b01} has a subsonic
solution satisfying
\begin{equation}
0\leq \psi \leq \bar \psi,\,\, \bar \psi-\psi \leq C\rho_0,\,\,\text{and}\,\, \sup_{x\in\Omega} \frac{|\nabla \psi|}{\gamma H^{\frac{\gamma+1}{2}}(|\nabla\psi, \psi;\rho_0)}\leq \frac{1}{4},
\end{equation}
where
\begin{equation}\label{defbarpsi}
\bar\psi=\rho_0\int_0^{x_2}u_0(s)ds
\end{equation}
and the constant $C$ depends only on $\max f(x_1)$ and $\max
u_0(x_2)$.
\end{prop}

The problem \eqref{b01} for subsonic flows is a Dirichlet problem
for a quasilinear elliptic equation. The domain $\Omega$ is
unbounded in both $x_1$ and $x_2$ directions so that the stream
function becomes an unbounded function, which is one of the main
differences from the problem of subsonic flows in infinitely long
nozzles. To overcome this difficulty, we first establish the
approximated problems in some infinitely long nozzles and then
obtain the existence of the subsonic solution in $\Omega$ by the
uniform estimates for the approximated solutions.

\subsection{Existence of subsonic solutions in nozzles}
Let $J=\sup f(x_1)$. Given $L\in \mathbb{N}$ satisfying $L>J$, define
 \be\label{b1} \O_L=\{x\in\O|f(x_1)<x_2< L\}
\ \  \text{ and}\ \ \Gamma_L=\{(x_1, L)|x_1\in \mathbb{R}\}. \ee

\begin{figure}[!h]
\includegraphics[width=130mm]{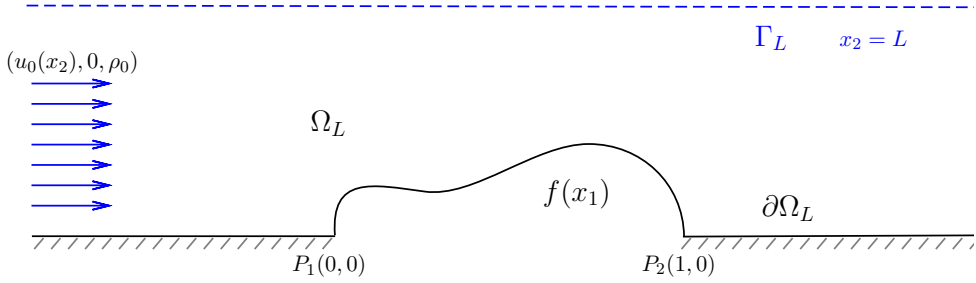}
\caption{Truncated domain $\O_L$}\label{f6}
\end{figure}

Let \be\label{b4} g_L(x_2)=\left\{\ba{ll}u_0'(x_2)&\text{if}\
\ x_2\leq L-1,\\
u_{0}'(L-1)(L-x_2)\ \ \ \ \ &\text{if}\ \ L-1<x_2\leq L,\ea\right.
\ee and $u_{0,L}(x_2)=u_0(0)+\int_0^{x_2}g_L(s)ds$. Hence, it
follows that for $x_2\in [0, L]$,
\[
u_{0,L}(x_2) \geq u_{0}(L-1) +\frac{u_0'(L-1)}{2}.
\]
If $L$ is sufficiently large, then $u_{0,L}(x_2) \geq \bar u/2$ for all $x_2\in [0, L]$.
Furthermore,  $u_{0,L}(x_2)$ satisfies $u_{0,L}'(L)=g_L(L)=0$ and
\be\label{b05}
u_{0,L}''(x_2)=g_L'(x_2)=\left\{\ba{ll}
u_0''(x_2)\quad\quad\quad &\text{if}\quad x_2\leq L-1,\\
-u_0'(L-1)&\text{if}\quad L-1<x_2\leq L.
\ea\right.
\ee
Therefore, $u_{0,L}''(x_2)\geq 0$ for $x_2\in [0,L]$.

Let $\kappa_L(\psi;\rho_0)$ satisfy
\begin{equation}\label{kappaL}
\psi=\int_{0}^{\kappa_L(\psi;\rho_0)}\rho_0u_{0,L}(s)ds.
\end{equation}
Denote
\begin{equation}\label{defFmL}
F_L(\psi;\rho_0)=u_{0,L}(\kappa_L(\psi;\rho_0)),\,\,{W}_L(\psi) ={F}_L(\psi;\rho_0) {F}_L'(\psi;\rho_0),\,\, m_L=: \int_0^L \rho_0 u_{0, L} (s) ds,
\end{equation}
and
\begin{equation}
{\mathcal{B}}_L(\psi;\rho_0) =h(\rho_0)+\f{1}{2}{F}_L^2(\psi;\rho_0)\quad  \text{and}\quad H_L(|\nabla \psi|^2, \psi;\rho_0) =\mclH(|\nabla \psi|^2,\B_L(\psi;\rho_0)).
\end{equation}
We consider the problem
\begin{equation}\label{Tbvp}
\left\{
\begin{aligned}
& \left(\left(1-\frac{|\nabla\psi|^2}{\gamma H_L^{\gamma+1}}\right)\delta_{ij} +\frac{\partial_i\psi\partial_j\psi}{\gamma H_L^{\gamma+1}}
\right)\partial_{ij}\psi ={W_L {H}_L^2}\quad \text{in}\,\,\Omega_L,\\
& \psi=0\quad \text{on}\,\, \Gamma, \quad \psi=m_L \quad \text{on}\,\, \Gamma_L.
\end{aligned}
\right.
\end{equation}
The problem \eqref{Tbvp} has been studied in detail in \cite{XX3,
DXX, DX}. For the convenience, we give a sketch for the study of
\eqref{Tbvp} in this subsection. There are several difficulties for
this problem. First, $\B_L$ is not well-defined if $\psi \notin [0,
m_L]$. Second, the equation in \eqref{Tbvp} is degenerate as
$\frac{|\nabla \psi|^2}{\gamma H_L^{\gamma+1}}$ approaches to $1$.
Finally, the problem is on an unbounded domain.

It is easy to see that
\begin{equation}
 F_L'(0;\rho_0)\leq 0 \quad \text{and}\quad {F}_L'(m_L;\rho_0)=0.
\end{equation}
If we extend ${F}_L$ to $\check{F}_L$ such that
\begin{equation}\label{defFL}
\check{F}_L(s;\rho_0)=\left\{
\begin{array}{ll}
{F}_L(s;\rho_0)\,\, &\text{if} \,\, s\in [0, m_L],\\
F_L(m_L;\rho_0)\,\, &\text{if}\,\, s>m_L,\\
{F}_L(0;\rho_0)+\frac{{F}_L'(0;\rho_0)}{2}\left(s+\frac{s^2}{2}\right)\,\, &\text{if}\,\, -1\leq s<0,\\
{F}_L(0;\rho_0)-\frac{{F}_L'(0;\rho_0)}{2}\,\, &\text{if}\,\, s<-1,
\end{array}
\right.
\end{equation}
then $\check{F}_L$ is a continuous function.  Define
\begin{equation}\label{defcWB}
\check{W}_L(\psi;\rho_0) =\check{F}_L(\psi;\rho_0) \check{F}_L'(\psi;\rho_0)\quad \text{and}\quad
\check{\mathcal{B}}_L(\psi;\rho_0) =h(\rho_0)+\f{1}{2}\check{F}_L^2(\psi;\rho_0).
\end{equation}
In order to deal with the possible degeneracy appeared in the
equation in \eqref{Tbvp}, as in \cite{XX3, DXX, DX}. we first
truncate the associated equation. Let $\chi(s)$ be a smooth
increasing odd function satisfying
\begin{equation}\label{defzeta}
\zeta(s)=\left\{
\begin{aligned}
&s\quad\,\,\,\,\,\, \text{if}\,\, |s|\leq 1/2,\\
&5/8\quad \text{if}\,\, s>3/4.
\end{aligned}
\right.
\end{equation}
Define ${\check{H}}_L(|\nabla \psi|^2, \psi; \rho_0)=\mclH\left(\zeta^2\left(\frac{|\nabla\psi|}{\Sigma(\cmB_L(\psi; \rho_0))}\right)\Sigma^2(\cmB_L(\psi; \rho_0)), \cmB_L(\psi;\rho_0)\right)$ where $\Sigma$ is the function defined in \eqref{defSigma}. Instead of \eqref{Tbvp}, we first study the following problem
\begin{equation}\label{Mbvp}
\left\{
\begin{aligned}
& \left(\left(1-\zeta^2\left(\frac{|\nabla\psi|}{\sqrt{\gamma}\cH_L^{\frac{\gamma+1}{2}}}\right)\right)\delta_{ij} +\zeta\left(\frac{\partial_i\psi}{\sqrt{\gamma}\cH_L^{\frac{\gamma+1}{2}}}\right)\zeta\left(\frac{\partial_j\psi}{\sqrt{\gamma}\cH_L^{\frac{\gamma+1}{2}}}\right)
\right)\partial_{ij}\psi  ={\cW_L\cH_L^2}\quad \text{in}\,\,\Omega_L,\\
& \psi=0\quad \text{on}\,\, \Gamma, \quad \psi=m_L \quad \text{on}\,\, \Gamma_L.
\end{aligned}
\right.
\end{equation}

Later on, a constant is said to depend on the elliptic coefficients,
if it depends on the coefficients of an elliptic equation. $C$
denotes a constant independent of $L$ and the elliptic coefficients,
$\msC$ denotes a constant which is independent of $L$ but depends on
the elliptic coefficients, and  $\mfC$ is a constant depending
both $L$ and elliptic coefficients. The value of these constants may
vary from line to line but they keep the same property.

\begin{lemma}\label{lemma31}
For any $\rho_0>\rho_0^*$, there exists a solution $\psi_L\in
C^{2,\a}(\O_L)\cap C^{1,\b}(\bar\O_L)$ to the problem \eqref{Mbvp}
satisfying \be\label{b6} 0\leq\psi_L\leq m_L \ \ \text{in}\ \ \O_L.
\ee
 Furthermore, there exists a $\check{\rho}_{0,L} \in (\rho_0^{*}, \infty)$ such that if $\rho_0> \check{\rho}_{0,L}$, it holds that
\begin{equation}
\max \frac{|\nabla \psi(x; \rho_0)|}{\Sigma(\B_L(\psi; \rho_0))}< \frac{1}{4}.
\end{equation}
\end{lemma}

The proof of this lemma is a combination of \cite[Section 2.2]{DX}
and  \cite[Proposition 3.1]{DXX}. We sketch the proof here.
\begin{proof}
Step 1. Solve the elliptic problem in bounded domains. Set
$\Omega_{L,k} =\Omega_L\cap \{(x_1, x_2)\big||x_1|\leq k\}$. It
follows from \cite[Theorem 12.5]{GT} that there is a solution
$\psi_{L,k}\in C^{2, \alpha}(\Omega_{L,k})\cap
C^0(\overline{\Omega_{L,k}})$ for the problem
\begin{equation}
\left\{
\begin{aligned}
& \left(\left(1-\zeta^2\left(\frac{|\nabla\psi|}{\sqrt{\gamma}\cH_L^{\frac{\gamma+1}{2}}}\right)\right)\delta_{ij} +\zeta\left(\frac{\partial_i\psi}{\sqrt{\gamma}\cH_L^{\frac{\gamma+1}{2}}}\right)\zeta\left(\frac{\partial_j\psi}{\sqrt{\gamma}\cH_L^{\frac{\gamma+1}{2}}}\right)
\right)\partial_{ij}\psi  = {\cW_L\cH_L^2}\quad \text{in}\,\,\Omega_{L,k},\\
& \psi=0\quad \text{on}\,\, \Gamma\cap \partial\Omega_{L,k}, \quad \psi= \int_0^{x_2} \rho_0 u_{0, L} (s) ds \quad \text{on}\,\, \partial\Omega_{L,k}\setminus \Gamma.
\end{aligned}
\right.
\end{equation}

Step 2. Since $\check{W}_L(s;\rho_0)\leq 0$ if $s\leq 0$ and $\check{W}_L(s;\rho_0)=0$ if $s\geq m_L$,  applying the maximum principle implies that
\[
0\leq \psi_{L, k}\leq m_L.
\]

Step 3. The H\"{o}lder gradient estimate \cite[Theorem 12.4]{GT} and the H\"{o}lder gradient estimate near the corners \cite[Section 2.2]{DX} show that
\begin{equation}\label{estpsilk}
\|\psi_{L, 2k}\|_{C^{1,\beta}(\overline{\Omega_{L,k}})} \leq \mfC \rho_0
\end{equation}
where $\mfC$ depends on $u_0$ and $\beta\in (0, \alpha)$ is a positive constant.

Step 4. The estimate \eqref{estpsilk}, together with Arzela-Ascoli lemma, implies that there exists a subsequence of $\{\psi_{L,k}\}$,
still labelled by $\{\psi_{L,k}\}$, converging to $\psi_L$ in any compact subset of $\Omega_L$.
Note that there exists a uniform constant $C$ such that
\[
\Sigma(\B_L(\psi_L;\rho_0))\geq C\rho_0^{\frac{\gamma+1}{2}}.
\]
Therefore,  there exists a $\check{\rho}_{0,L}$ such that if $\rho_0>\check{\rho}_{0,L}$, one has
\[
\frac{|\nabla \psi_L|}{\Sigma(\B_L(\psi_L;\rho_0))}\leq \frac{\mfC\rho_0}{C\rho_0^{\frac{\gamma+1}{2}}}=\frac{\mfC}{C}\rho_0^{\frac{1-\gamma}{2}} \leq  \frac{1}{4}.
\]
This finishes the proof of the lemma.
\end{proof}

The rest of this section is devoted to prove that $\check{\rho}_{0,
L}$ is indeed independent of $L$, so that there exist subsonic flows
in $\Omega$ as long as $\rho_0$ is suitably large, which yields the
proof of Proposition \ref{prop31}.
\subsection{Uniform estimates for subsonic flows in nozzles}
Since the stream function is unbounded as $x_2\to \infty$, in order
to get uniform estimates for the stream function, we study the
difference of the stream function from the one in the upstream.

\begin{prop}
For any $\epsilon>0$, there exists an $\bar L>0$ such that if $L>\bar L$, then for any $\rho_0\in ((1+\epsilon)\rho_0^*, \infty)$, there exists a unique
triple $(\chi(s), \rho_{1, L}, u_{1,L})$ satisfying that
\begin{enumerate}
\item $\rho_{1,L}$ is a constant, $\chi$ is an increasing function maps $[0,L]$ to $[J, L]$, and $u_{1,L}$ is a function on $[J,
L]$, where $J=\sup f(x_1)$;
\item the triple $(\chi(s), \rho_{1, L}, u_{1, L})$ satisfies
 \be\label{g1}
 0<\rho_{1,L}< \rho_0,\quad u_{1,L}(x_2)
> 0\quad\text{and}\quad\max_{J\leq x_2\leq L}u_{1,L}(x_2)<
\sqrt{\gamma} \rho_{1,L}^{\frac{\gamma+1}{2}};\ee
\item
 for $s\in[0,L]$, $(\chi(s), \rho_{1,L}, u_{1,L})$ satisfies
\be\label{b46}
\f{u_{0,L}^2(s)}2+h(\rho_0)=\f{u_{1,L}^2(\chi(s))}2+h(\rho_{1,L})
\ee and \be\label{b47}
\int_0^s\rho_0u_{0,L}(t)dt=\int_J^{\chi(s)}\rho_{1,L}u_{1,L}(t)dt.
\ee
\end{enumerate}
Furthermore, let $\bar\psi_L$ and $\hat\psi_L$ be defined as follows,
\begin{equation}\label{defhpsiL}
\bar\psi_L(x_2)=\rho_0\int_0^{x_2}u_{0,L}(s)ds\quad \text{and}\quad \hat\psi_L(x_2)=\rho_{1,L}\int_J^{x_2}u_{1,L}(s)ds.
\end{equation}
Then, for $x_2\in [J, L]$, it holds that
\begin{equation}\label{estbhpsi}
0\leq \bar\psi(x_2)-\hat\psi(x_2) \leq C \rho_0,
\end{equation}
where $C$ is a uniform constant depending only on $J$ and $\max u_0$.
\end{prop}
\begin{proof}
Define  $\underline\varrho_{1,L}$ and $\bar\varrho_{1,L}$ to be the
constants satisfying \be\label{b49}
h(\bar\varrho_{1,L})=\f12\min u_{0,L}^2(x_2) +h(\rho_0) \ \ \
\text{and}\ \ \ \
\f12 \gamma \underline\varrho_{1,L}^{\gamma-1}+h(\underline\varrho_{1,L})=\f12\max
u_{0,L}^2(x_2) +h(\rho_0). \ee If $\rho_0>\rho_0^*$, then it is
easy to see that $\underline \varrho_{1,L}<\rho_0<\bar\varrho_{1,L}$.
Furthermore, for any $\rho\in \left(\underline{\varrho}_{1,L},
\bar{\varrho}_{1,L}\right)$, one has
$$\min_{s\in [0,L]}D(s;\rho)>0\ \ \ \ \text{and}\ \ \ \ \ \max_{s\in[0,L]} \sqrt{D(s;\rho)} < \sqrt{\gamma}\rho^{\frac{\gamma-1}{2}} \quad \text{for} \quad s\in [0,1],
$$
where $D(s;\rho)=2(h(\rho_0)-h(\rho))+u_{0,L}^2(s)$.

Differentiating \eqref{b47} with respect to $s$ and substituting
\eqref{b46} into the associated equation yield that \be\label{b48}
\f{d\chi}{ds}=\f{\rho_0u_{0,L}(s)}{\rho_{1,L}\sqrt{D(s;\rho_{1,L})}}\
 \ \  \text{for}\ \ s\in[0,L], \,\, \text{and}\, \,  \chi(0)=J.\ee
Hence, it suffices to show that there exists an $\bar L>0$ such that if $L>\bar L$, then for any $\rho_0 \in ((1+\epsilon)\rho_0^*, \infty)$  there exists a unique $\rho_{1,L}\in\left(\underline\varrho_{1,L},\bar\varrho_{1,L}\right)$ such that
\be\label{b50}
\int_0^L\f{\rho_0u_{0,L}(s)}{\rho_{1,L}\sqrt{D(s;\rho_{1,L})}}ds=L-J.
\ee
  Once  $\rho_{1,L}$ is determined, then $\chi(s)$ follows from \eqref{b48} and  $u_1(\chi(s))=\sqrt{D(s; \rho_{1,L})}$.

Let
$G(\rho)=\int_0^L\f{\rho_0u_{0,L}(s)}{\rho\sqrt{D(s;\rho)}}ds$.
A direct computation yields that
\begin{equation}\label{beq2-9}
G'(\rho)=\int_0^L\f{\rho_0u_{0,L}(s)}{\rho^2D^{3/2}(s;\rho)}\left(\gamma \rho^{\gamma-1}-D(s;\rho)\right)ds,
\end{equation}
which is always positive for $\rho\in
\left(\underline\varrho_{1,L},\bar\varrho_{1,L}\right)$. Thus
$G(\rho)$ is a strictly increasing function in
$\left(\underline\varrho_{1,L},\bar\varrho_{1,L}\right)$. Therefore, the
range of $G(\rho)$ on $(\underline\varrho_{1,L}, \bar\varrho_{1,L})$
is $\left(G(\underline\varrho_{1,L}), G(\bar\varrho_{1,L})\right)$.
We claim that $L-J$ lies in the interval
$\left(G(\underline\varrho_{1,L}), G(\bar\varrho_{1,L})\right)$ if $L$ is sufficiently large.

First, it is easy to see that $G(\rho_0) =L>L-J$. It follows from
the definition of $\underline\varrho_{1,L}$ in \eqref{b49} that
\be\label{g3}
\begin{aligned}
&G(\underline\varrho_{1,L})=\int_0^L\f{\rho_0u_{0,L}(s)}{\underline\varrho_{1,L}\sqrt{\gamma \underline\varrho_{1,L}^{\gamma-1}+u^2_{0,L}(s)-\max
u_{0,L}^2(s)}}ds\\
 =&
\int_0^L\f{\rho_0u_{0,L}(s)}{\left(\f{\gamma-1}{\gamma(\gamma+1)}\max
u_{0,L}^2(s)+\f2{\gamma+1}\rho_0^{\gamma-1}\right)^{\f1{\gamma-1}}\sqrt{\frac{2}{\gamma+1}(\gamma \rho_0^{\gamma-1} -\max u_{0,L}^2) +u_{0,L}^2(s)}}ds\\
=&\int_0^L\f{1}{\left(\f{\gamma-1}{\gamma+1}\frac{\max
u_{0,L}^2(s)}{\gamma \rho_0^{\gamma-1}}+\f2{\gamma+1}\right)^{\f1{\gamma-1}}\sqrt{\frac{2}{\gamma+1}(\frac{\gamma \rho_0^{\gamma-1}}{\max u_{0,L}^2}-1) \frac{\max u_{0,L}^2} {u_{0,L}^2(s)}+1}}ds.
\end{aligned}
\ee
When $\rho_0>\rho_0^*=\left(\frac{\max u_{0,L}^2}{\gamma}\right)^{\frac{1}{\gamma-1}}$, one has
\be
G(\underline{\varrho}_{1,L})\leq
\int_0^L\f{1}{\left(\f{\gamma-1}{\gamma+1}\frac{\max
u_{0,L}^2(s)}{\gamma \rho_0^{\gamma-1}} +\f2{\gamma+1}\right)^{\f1{\gamma-1}}\sqrt{\frac{2}{\gamma+1}(\frac{\gamma \rho_0^{\gamma-1}}{\max u_{0,L}^2}-1) +1}}ds.
\ee
Define
\[
\G(\tau)= \left(\f{\gamma-1}{\gamma+1}\tau+\f2{\gamma+1}\right)^{\f1{\gamma-1}}\sqrt{\frac{2}{\gamma+1}\left(\frac{1}{\tau}-1\right) +1}.
\]
The straightforward computations yield
\begin{equation}
\G'(\tau)= \left(\f{\gamma-1}{\gamma+1}\tau+\f2{\gamma+1}\right)^{\f1{\gamma-1}}\sqrt{\frac{2}{\gamma+1}\left(\frac{1}{\tau}-1\right) +1}\frac{\tau-1}{\tau(2+\tau(\gamma-1))},
\end{equation}
so $\G(\tau)$ is a decreasing function on $[0,1]$. Since $\G(1)=1$,
for any $\epsilon>0$, there exists an $\bar L>0$ such that if
$L>\bar L$, it holds that
\[
\G\left(\frac{1}{(1+\epsilon)^{\gamma-1}}\right)\geq \frac{L}{L-J}.
\]
Therefore, for any $\rho_0>(1+\epsilon)\rho_0^*$, if $L>\bar L$,
\begin{equation}
G(\underline{\varrho}_{1,L})\leq \int_0^L \frac{1}{\G\left(\frac{\max u_{0,L}^2(s)}{\gamma \rho_0^{\gamma-1}}\right)}ds\leq \int_0^L \frac{1}{\G\left(\frac{1}{(1+\epsilon)^{\gamma-1}}\right)}ds\leq \int_0^L\frac{L-J}{L} ds=L-J.
\end{equation}
Hence, for any $\rho_0\in ((1+\epsilon)\rho_0^*, \infty)$,  there exists a unique $\rho\in (\underline{\varrho}_{1,L}, \rho_0)$ satisfying $G(\rho)=L-J$.

Define $\bar \chi(s)= \chi(s)-s$. It follows from \eqref{b48} that
\begin{equation}
\bar\chi'(s) =\frac{\rho_0 u_{0, L}(s) -\rho_{1,L} u_{1, L}(\chi(s))}{\rho_{1,L} u_{1,L}(\chi(s))}.
\end{equation}
As mentioned in Section \ref{secform}, for given $\mathfrak{s}$ in
\eqref{eq2Ber}, $\mathfrak{m}$ is strictly decreasing with respect
to $\rho$ in the subsonic region. Since $\rho_{1,L}<\rho_0$,
one has $\rho_0 u_{0,L}(s) < \rho_{1,L}u_{1,L}(\chi(s))$. Therefore,
$\bar\chi'(s)<0$. Hence, $\bar\chi(L) \leq \bar\chi(s)\leq
\bar\chi(0)$. This gives $0 \leq \chi(s)-s\leq J$ which is
equivalent to $0 \leq x_2-\chi^{-1}(x_2)\leq J$ for $x_2\in [J,L]$.
If $x_2\in [J, L]$, then direct computations yield
\begin{equation}
\begin{aligned}
\bar\psi(x_2)-\hat\psi(x_2) = &\int_0^{x_2} \rho_0 u_{0,L}(s)ds -\int_J^{x_2} \rho_{1,L}u_{1,L}(s)ds\\
= & \int_0^{x_2} \rho_0 u_{0,L}(s)ds -\int_{\chi^{-1}(J)}^{\chi^{-1}(x_2)} \rho_{1,L}u_{1,L}(\chi(s))\chi'(s)ds\\
=&\int_0^{x_2} \rho_0 u_{0,L}(s)ds -\int_{0}^{\chi^{-1}(x_2)} \rho_{0,L}u_{0,L}(s)ds\\
=&\int_{\chi^{-1}(x_2)}^{x_2}\rho_0 u_{0, L}(s) ds,
\end{aligned}
\end{equation}
where \eqref{b48} has been used in the third equality. Since $0\leq
x_2- \chi^{-1}(x_2)\leq J$, one gets
\begin{equation}
0\leq \bar\psi(x_2)-\hat\psi(x_2)\leq C \rho_0,
\end{equation}
where the constant $C$ depends only on $J$ and $\max u_0$.
This finishes the proof for the proposition.
\end{proof}

One of our key estimates is the following upper and lower bounds
estimate for $\psi_L$, which also plays an important role in proving  the uniform gradient
estimate for $\psi_L-\bar \psi_L$.
\begin{prop}\label{proplinfty}
Let $\hat\Omega_L=\{(x_1, x_2)|x_2\in (J,L), x_1\in \mathbb{R}\}$.
For any $\epsilon>0$, there exists an $\bar L>0$ such that if
$L>\bar L$, then for any $\rho_0\in ((1+\epsilon)\rho_0^*, \infty)$,
if $\psi_L$ is a subsonic solution of the problem \eqref{Tbvp}, then
it holds that
\begin{equation}\label{linftyest}
\psi_L\geq  \hat{\psi}_L\,\,\text{in}\,\, \hat{\Omega}_L\,\,\text{and}\,\, 0\leq \psi_L \leq \bar{\psi}_L\,\,{\text
in}\,\,\O_L,
\end{equation}
where $\hat\psi_L$ and $\bar\psi_L$ are defined in \eqref{defhpsiL}.
\end{prop}
\begin{proof}
Define
\begin{equation}\label{b9}
\Psi_L(x)=\psi_L(x)-\bar\psi_L(x_2)\,\, \text{for}\,\, x\in \Omega\,\, \text{and}\,\, \hat{\Psi}_L=\hat{\psi}_L-\psi_L\,\,  \text{for}\,\, x\in\hat\Omega_L.
\end{equation}
Note that $\hat\psi_L$ satisfies the equation \eqref{Tbvp} with the
boundary conditions $\hat \psi_L=0$ at $x_2=J$ and $\hat\psi_L=m_L$
at $x_2=L$. Then straightforward computations show that $\hat\Psi_L$
solves the following problem \be\label{sup1} \left\{\ba{ll}
\p_i\left(a_{ij}(\nabla \psi,\psi;\nabla \hat\psi, \hat\psi)\p_j\hat{\Psi}_L+b_i(\nabla \psi,\psi;\nabla \hat\psi,\hat\psi)\hat{\Psi}_L\right)\\
=b_i(\nabla \psi,\psi;\nabla \hat\psi, \hat\psi)\p_i\hat{\Psi}_L+d(\nabla \psi,\psi; \nabla \hat\psi, \hat\psi)\hat{\Psi}_L\quad\text{in}\quad\hat{\O}_L,\\
\hat{\Psi}_L\leq 0\,\,\text{if}\,\, x_2=J,\,\, \text{and}\,\, \hat{\Psi}_L=0\,\, \text{if}\,\, x_2=L,\ea\right. \ee
where the Einstein summation convention is used and $a_{ij}$, $b_i$, and $d$ are defined as follows,
\be\label{c25}
a_{ij}(\nabla\psi,\psi; \nabla \hat\psi, \hat\psi)=\int_{0}^{1}\f{1}{H_{L}(|\nabla \hat\psi_{L,t}|^2, \hat\psi_{L,t})}\left(\d_{ij}+\f{\p_i\hat\psi_{L,t}\p_j\hat\psi_{L,t}}{\gamma H_{L}^{\gamma+1}(|\nabla \hat\psi_{L,t}|^2, \hat\psi_{L,t})-|\nabla\hat\psi_{L,t}|^2)}\right)dt,\ee
\be\label{c26}
\begin{aligned}
&b_i(\nabla\psi,\psi;\nabla\hat\psi, \hat\psi)=-\int_{0}^{1}\f{\p_i\hat\psi_{L,t}
F_L(\hat\psi_{L,t})F_L'(\hat\psi_{L,t})H_{L}(|\nabla \hat\psi_{L,t}|^2, \hat\psi_{L,t})}{\gamma H_{L}^{\gamma+1}(|\nabla \hat\psi_{L,t}|^2, \hat\psi_{L,t})-|\nabla\hat\psi_{L,t}|^2}dt,
\end{aligned}
\ee
and
\be\label{c28}
\begin{aligned}
d(\nabla\psi,\psi;\nabla\hat\psi, \hat\psi)=&\int_{0}^{1}\f{H_{L}^3(|\nabla \hat\psi_{L,t}|^2, \hat\psi_{L,t})\left(F_L(\hat\psi_{L,t})F_L'(\hat\psi_{L,t})\right)^2}{\gamma H_{L}^{\gamma+1}(|\nabla \hat\psi_{L,t}|^2, \hat\psi_{L,t})-|\nabla\hat\psi_{L,t}|^2}dt\\
&+ \int_{0}^{1}H_{L}(|\nabla \hat\psi_{L,t}|^2, \hat\psi_{L,t})\left(F_L(\hat\psi_{L,t})F_L''(\hat\psi_{L,t})+(F_L'(\hat\psi_{L,t}))^2\right)dt\\
=&\sum_{i=1}^2 d_i(\nabla\psi,\psi;\nabla\hat\psi, \hat\psi),
\end{aligned}
\ee {for} $i$, $j=1$, $2$, with
$\hat\psi_{L,t}=t\hat\psi_L+(1-t){\psi}_L$ for $t\in[0,1]$. Here and
in what follows we neglect the parameter $\rho_0$ in the
coefficients when there is no confusion.

Set $\hat{\Psi}^{+}_{L}=\max\left\{\hat{\Psi}_{L},0\right\}$. Multiplying the equation in (\ref{sup1}) with $\hat\Psi_L^+$ and integrating by parts imply that
\be\label{sup3}\ba{ll}
&\iint_{\hat{\O}_L}\left[\left|\nabla \hat{\Psi}^{+}_{L}\right|^2\int_{0}^{1}\f{1}{H_{L}(|\nabla \hat\psi_{L,t}|^2, \hat\psi_{L,t})}dt \right.\\
&\left. +\int_{0}^{1}\f{\left(H_{L}^2(|\nabla \hat\psi_{L,t}|^2, \hat\psi_{L,t}) F_L(\hat\psi_{L,t})F_L'(\hat\psi_{L,t})\hat{\Psi}^{+}_{L}-\nabla\tpsi\cdot \nabla \hat{\Psi}_{L}^{+}\right)^2}{H_{L}(|\nabla \hat\psi_{L,t}|^2, \hat\psi_{L,t})(\gamma H_{L}^{\gamma+1}(|\nabla \hat\psi_{L,t}|^2, \hat\psi_{L,t})-|\nabla\hat\psi_{L,t}|^2)}dt\right]dx_1dx_2\\
=&-\iint_{\hat{\O}_L}d_2(\nabla\psi,\psi;\nabla\hat\psi,
\hat\psi)\left(\hat{\Psi}^{+}_{L}\right)^2dx_1dx_2. \ea \ee Then
direct computations yield that
$$
F_L(\hat\psi_{L,t})F_L''(\hat\psi_{L,t})+\left(F_L'(\hat\psi_{L,t})\right)^2=\f{u_{0,L}''(\kappa_L(\hat\psi_{L,t};\rho_0))}{\rho_0u_{0,L}(\kappa_L(\hat\psi_{L,t};\rho_0))}\geq0,
$$
hence $d_2\geq 0$ for subsonic flows with the assumption \eqref{a2}.
Therefore, it follows from \eqref{sup3} that
$$\nabla \hat{\Psi}^{+}_{L}=0\quad\quad\text{in}\quad \hat\Omega_L.$$
Since $\hat{\Psi}^{+}_{L}=0$ on $\p\hat{\O}_L$, one has
$\hat{\Psi}^{+}_{L}=0$ in $\hat\O_L$.
Thus  $\hat{\Psi}_L\leq 0$ in $\hat{\O}_L$ which is equivalent to
$\psi_L\geq\hat{\psi}_L$ in $\hat{\O}_L$.

Similarly, one can show that $\psi_L\leq\bar{\psi}_L$ in $\O_L$.
This finishes the proof of the proposition.
\end{proof}

\begin{corollary}\label{cor35}
For any $\epsilon>0$, there exists an $\bar L>0$ such that if
$L>\bar L$, then for any $\rho_0\in ((1+\epsilon)\rho_0^*, \infty)$,
if $\psi_L$ is a subsonic solution of the problem \eqref{Tbvp}
satisfying $0\leq \psi_L\leq m_L$, then it holds that
\begin{equation}\label{eq342}
-C\rho_0\leq \Psi_L \leq 0
\end{equation}
where the constant $C$ depends only on $J$ and $\max u_0$.
\end{corollary}
\begin{proof}
First, it follows from \eqref{linftyest} that $\Psi_L\leq 0$. If $x_2\geq J$, then
\[
\Psi_L =\psi_L-\bar\psi_L\geq \hat\psi_L-\bar\psi_L\geq -C\rho_0,
\]
where \eqref{estbhpsi} has been used. If $x_2\in (0, J)$, then
\eqref{linftyest} implies that $\psi_L\geq 0$. Therefore,
\[
\Psi_L \geq \psi_L-\bar\psi_L\geq -\bar\psi_L \geq -C \rho_0.
\]
This proves the corollary.
\end{proof}

As a consequence of the elliptic estimate, one has the following
lemma.
\begin{lemma}\label{lem36}
For any $\epsilon>0$, there exists an $\bar L>0$ such that if $L>\bar L$, then for any $\rho_0\in ((1+\epsilon)\rho_0^*, \infty)$, if $\psi_L$ is a subsonic solution of the problem \eqref{Tbvp} satisfying $0\leq \psi_L\leq m_L$, then
\begin{equation}\label{hgPsi}
\|\Psi_L\|_{C^{1,\beta}(\overline{\Omega_L})}\leq \msC\rho_0
\end{equation}
with $\beta\in(0, \alpha)$ independent of $L$.
\end{lemma}
\begin{proof}
Note that $\Psi_L$ satisfies the problem
\begin{equation}\label{pbPsiL}
\left\{
\begin{aligned}
&\left(\left(1-\frac{|\nabla \psi_L|^2}{\gamma H_L^{\gamma+1}}\right)\delta_{ij} +\frac{\partial_i\psi_L\partial_j\psi_L}
{\gamma H_L^{\gamma+1}}\right)\partial_{ij}\Psi_L = {F_LF_L'(H_L^2-\rho_0^2 -\frac{(\partial_1 \psi_L)^2}{\gamma H_L^{\gamma+1}}\rho_0^2)} \quad \text{in}\,\, \Omega_L,\\
&\Psi_L =-\int_0^{f(x_1)}\rho_0 u_{0, L}(s) ds\,\, \text{on}\,\, \Gamma,\quad \Psi_L=0\,\, \text{on}\, \, \Gamma_L.
\end{aligned}
\right.
\end{equation}
It follows from the H\"{o}lder gradient estimate \cite[Theorem 12.4]{GT} and the estimate near the corners \cite[Section 2.2]{DX} for the elliptic equations that \eqref{hgPsi} holds.
\end{proof}

In fact, we can also prove the following far field behavior and uniform integral estimate for $\Psi_L$.
\begin{lemma}\label{lem37}
Suppose that $u_0''(x_2)\geq 0$.  If $\psi_L$ is a subsonic solution of the problem \eqref{Tbvp}, then we have
\begin{equation}\label{psiLfar}
|\psi_L -\bar\psi_L|\to 0 \quad \text{uniformly with respect to }x_2\in [0, L] \text{ as}\,\, |x_1|\to \infty
\end{equation}
and
 \be\label{d3}
\|\nabla(\psi_L-\bar\psi_L)\|_{L^2(\O_L)}\leq \msC \ee where the constant $\msC$ depends on elliptic coefficients and is independent of $L$. Furthermore, the solution of the problem \eqref{Tbvp} is unique.
\end{lemma}
\begin{proof}
First, the far field behavior \eqref{psiLfar} and the uniqueness of
the solution follow from  \cite[Lemma 4.1]{DXX} and
\cite[Proposition 5]{XX3}.

Note that $\Psi_L$ satisfies the following
problem \be\label{L2-2} \left\{\ba{ll}
\p_i\left(a_{ij}(\nabla\psi,\psi; \nabla \bar\psi, \bar\psi)\p_j\Psi_L+b_i(\nabla\psi,\psi;\bar\psi, \bar\psi)\Psi_L\right)\\
\quad =b_i(\nabla\psi,\psi;\nabla\bar\psi, \bar\psi)\p_i\Psi_L+d(\nabla\psi,\psi;\nabla\bar\psi, \bar\psi)\Psi_L\quad\text{in}\quad\O_L,\\
\Psi_L=-\bar{\psi}_L \,\,\text{if}\,\, x_2=
f(x_1),\,\,\text{and}\,\, \Psi_L=0\,\,\text{if}\,\, x_2=L,\ea\right.
\ee where the coefficients $a_{ij}$, $b_i$, $d$ are defined as
in \eqref{c25}-\eqref{c28}.

Multiplying the equation in \eqref{L2-2} with $\Psi_L$ and
integrating in the domain $\O_{L,N}=\O_L\cap\{|x_1|\leq N\}$ yield
that \be \ba{lll}
&\ \ \ \ \iint_{\O_{L,N}}\left(a_{ij}\partial_i\Psi_L\partial_j\Psi_L+2b_i\Psi_L\partial_i\Psi_L+ d_1\Psi^2_L\right)dx_1dx_2\\
&
=\int_{\partial\O_{L,N}}(a_{ij}\partial_j\Psi_L+b_i\Psi_L)\Psi_Ln_idS
-\iint_{\O_{L,N}}d_2\Psi^2_Ldx_1dx_2\\
&=-\int_{\{x_2=f(x_1)\}}(a_{ij}\partial_j\Psi_L+b_i\Psi_L)\bar{\Psi}_Ln_idS-\iint_{\O_{L,N}}d_2\Psi^2_Ldx_1dx_2\\
&\ \ \ -\int_{\{x_1=
N\}}(a_{ij}\partial_j\Psi_L+b_i\Psi_L)\Psi_L n_idS+\int_{\{x_1=-
N\}}(a_{ij}\partial_j\Psi_L+b_i\Psi_L)\Psi_L n_idS \ea \ee where
$n=(n_1,n_2)$ is the unit normal to the boundary $\p\O_{L,N}$. Thus
\be\label{L2-0} \ba{lll} &\ \ \ \
\iint_{\O_{L,N}}\left(a_{ij}\partial_i\Psi_L\partial_j\Psi_L+2b_i\Psi_L\partial_i\Psi_L
+d_1\Psi^2_L\right)dx_1dx_2\\
&\leq
\msC\left(\int_{\R}|\bar{\psi}_L(f(x_1))|dx_1+\int_{f(N)}^{L}|\Psi_L(N,x_2)|dx_2
+\int_{f(-N)}^{L}|\Psi_L(-N,x_2)|dx_2\right). \ea
\end{equation}
 Here one has used the uniform bounds for $C^1$-norm of $\Psi_L$
and the positivity of $d_2$.

Then direct computations give \be\label{c14}
\begin{aligned}
& a_{ij}\partial_i\Psi_L\partial_j\Psi_L+2b_i\Psi_L\partial_i\Psi_L+d_1\Psi^2_L\\
= & \int_0^1\f{1}{H_{L}(|\nabla \bar\psi_{L,t}|^2, \bar\psi_{L,t})}\left(|\nabla\Psi_L|^2+\f{\left(\nabla\psi_t\cdot
\nabla\Psi_L-FF'\Psi_L H_L^2(|\nabla \bar\psi_{L,t}|^2, \bar\psi_{L,t})\right)^2}{\gamma H_{L}^{\gamma+1}(|\nabla\bar\psi_{L,t}|^2, \bar\psi_{L,t})-|\nabla\psi_t|^2}\right)dt,
\end{aligned}
\ee
where $\bar \psi_{L,t} =t\bar\psi_L+(1-t)\psi_L$ for $t\in [0,1]$.
Substituting \eqref{c14} into \eqref{L2-0} yields that
for any $N>0$, one has
\be\label{c15}
\begin{aligned}
\iint_{\O_{L,N}}|\nabla\Psi_L|^2dx_1dx_2\leq&
\msC\left(\int_{\mathbb{R}}|\bar{\psi}_L(x_1, f(x_1))|dx_1+\int_{f(N)}^{L}|\Psi_L(N,x_2)|dx_2\right.\\
&\quad \left.+\int_{f(-N)}^{L}|\Psi_L(-N,x_2)|dx_2\right).
\end{aligned}
\ee

The asymptotic behavior \eqref{psiLfar} implies that the second term on the right-hand side of
\eqref{c15} tends to zero as $N$ goes to infinity. Moreover, thanks
to the asymptotic assumption to the nozzle wall $x_2=f(x_1)$,  the first term on the right-hand side of \eqref{c15} is
uniformly bounded (independent of $L$).  Hence, we have \be
\iint_{\O_{L}}|\nabla\Psi_L|^2dx_1dx_2\leq
\msC\int_{\mathbb{R}}\left|\bar{\psi}_L(f(x_1))\right|dx_1 \leq \msC. \ee
This finishes the proof of the lemma.
\end{proof}
Given $\rho_0\in (\rho_0^*, \infty)$, let $\mathcal{S}_L(\rho_0)$ be the set of all solutions of the problem \eqref{Tbvp} associated with $\rho_0$. Define
\begin{equation}
\mathcal{M}_L(\rho_0) =\sup_{\psi_L\in \mathcal{S}_L(\rho_0)} \sup_{x\in \Omega_L}\frac{|\nabla\psi_L(x;\rho_0)|}{\Sigma(\B_L(\psi_L; \rho_0))}.
\end{equation}
Set
\begin{equation}
\bar\rho_{0,L} = \inf\left\{s|  \text{for any }\rho_0>s,
\mathcal{M}_L(\rho_0) < \frac{1}{4}\right\} \quad \text{and}\quad
\rho^{**}= (20\max u_0^2)^{\frac{1}{\gamma-1}}.
\end{equation}
\begin{lemma}\label{lem38}
If $\bar\rho_{0, L}>\rho^{**}$, then $\M(\bar\rho_{0,L})= \frac{1}{4}$.
\end{lemma}
\begin{proof}
First, it follows from the continuous dependence on the parameters
for solutions of uniformly elliptic equations that
$\M_L(\bar\rho_{0,L}) \leq \frac{1}{4}$. If $\M_L(\bar{\rho}_{0,L})<
1/4$ and $\bar\rho_{0,L}> \rho^{**}$, then it is easy to see that
there exists a $\delta>0$ such that $\M_L(s)\leq 1/4$ for $s\in
(\bar{\rho}_{0,L}-\delta, \bar{\rho}_{0,L})$. There is a
contradiction. So the proof of the lemma is completed.
\end{proof}

Now we have the following lemma.
\begin{lemma}
There exists a $\bar \rho_0\in (\rho_0^*, \infty)$ independent of $L$ such that if $\rho_0>\bar\rho_0$, there exists a subsonic solution of \eqref{Tbvp} satisfying
\begin{equation}
\frac{|\nabla \psi_L|}{\sqrt{\gamma} H_L^{\frac{\gamma+1}{2}}(|\nabla \psi_L|^2, \psi_L)}\leq \frac{1}{4}.
\end{equation}
\end{lemma}
\begin{proof}
If $\bar\rho_{0,L}>\rho^{**}$, then it follows from Lemma \ref{lem38} that the problem \eqref{Tbvp} associated with $\rho_0=\bar\rho_{0,L}$ has a solution $\psi_L$ satisfying
\begin{equation}
\sup_{\Omega_L}\frac{|\nabla \psi_L|}{\Sigma(\B_L(\psi_L;\rho_0)} =\frac{1}{4}.
\end{equation}
It follows from Lemma \ref{lem36} and the definition of $\B_L$ that one has
\begin{equation}
\frac{1}{4} \leq \frac{\sup |\nabla \psi_L|}{\inf \Sigma(\B(\psi_L;\rho_0)}\leq \frac{\msC \bar\rho_{0, L}}{C\bar\rho_{0, L}^{\frac{\gamma+1}{2}}}=\msC^\sharp \rho_0^{\frac{1-\gamma}{2}},
\end{equation}
where $\msC^\sharp$ is a constant independent of $L$. Therefore,
\begin{equation}
\bar\rho_{0, L} \leq (4\msC^\sharp)^{\frac{2}{\gamma-1}}.
\end{equation}
Choose
\begin{equation}
\bar\rho_0=\max\left(\rho^{**},
(4\msC^\sharp)^{\frac{2}{\gamma-1}}\right).
\end{equation}
It is easy to see that $\bar{\rho}_0$ is independent of $L$ and that if $\rho_0>\bar{\rho}_0$, the problem \eqref{Tbvp} has a solution $\psi_L$ satisfying $\frac{|\nabla \psi_L|}{\Sigma(\B(\psi_L;\rho_0))}\leq 1/4$. This finishes the proof of the lemma.
\end{proof}

\subsection{Existence of subsonic solutions with large incoming density}
If $\rho_0>\bar\rho_0$, then the problem \eqref{Tbvp} has a solution $\psi_L$ satisfying
\[
\|\psi_L-\bar\psi_L\|_{C^1(\Omega)}\leq \msC\rho_0\quad
\text{and}\quad \|\nabla (\psi_L-\bar\psi_L)\|_{L^2(\Omega)}\leq
\msC,\quad \text{and}\quad \sup_{\Omega_L}\frac{|\nabla
\psi_L|}{\Sigma(\B_L(\psi_L;\rho_0))} \leq \frac{1}{4}
\]
where $\msC$ is a constant independent of $L$. Let $L\to \infty$,
then there exists a subsequence of $\{\psi_L\}$ still labelled by
$\{\psi_L\}$, converging to $\psi$ satisfying \be\label{t1}
\|\psi-\bar\psi\|_{C^1(\Omega)}\leq \msC\rho_0\quad \text{and}\quad
\|\nabla (\psi-\bar\psi)\|_{L^2(\Omega)}\leq \msC,\quad
\text{and}\quad
\sup_{\Omega}\frac{|\nabla\psi|}{\Sigma(\B(\psi;\rho_0))}\leq
\frac{1}{4} \ee where the constant $\msC$ is a uniform constant.
Hence $\psi$ is a subsonic solution of the problem \eqref{b01}.
This finishes the proof for Proposition \ref{prop31}.

\section{Fine properties of the subsonic solutions past a wall}\label{secfine}

In this section, we study properties of subsonic solutions for
\eqref{b01} constructed in previous section, such as asymptotic
behaviors and positivity of horizontal velocity of subsonic flows.
\subsection{Asymptotic behavior at the far fields}
We claim that \be\label{c23} |\nabla\Psi(x_1,x_2)|\rightarrow
0\quad\text{as}\quad |x|\rightarrow +\infty,\ee where
$\Psi=\psi-\bar\psi$ with $\bar\psi$ defined in \eqref{defbarpsi}.
Note that the estimate \eqref{c23} gives the asymptotic behavior of
subsonic flows in the far fields. It follows from the definition
that $\Psi$ satisfies the equation
\begin{equation}
\left(\left(1-\frac{|\nabla \psi|^2}{\gamma H^{\gamma+1}}\right)\delta_{ij} +\frac{\partial_i\psi\partial_j\psi}
{\gamma H^{\gamma+1}}\right)\partial_{ij}\Psi ={F F'(H^2-\rho_0^2 -\frac{(\partial_1 \psi)^2}{\gamma H^{\gamma+1}}\rho_0^2)} \quad \text{in}\,\, \Omega.
\end{equation}
Since $\|\Psi\|_{L^\infty(\bar\Omega})\leq \msC\rho_0$, it follows
from the H\"{o}lder gradient estimate that
\begin{equation}\label{eq44}
\|\Psi\|_{C^{1, \beta}(\bar\Omega)}\leq \msC\rho_0.
\end{equation}

Now we prove \eqref{c23} by contradiction argument. If
(\ref{c23}) is false, for any positive constant $\e_0$, there exists
a sequence
$\left\{x^{(i)}=\left(x_1^{(i)},x_2^{(i)}\right)\right\}_{i=1}^\infty$
going to infinity such that
$\left|\nabla\Psi\left(x_1^{(i)},x_2^{(i)}\right)\right|\geq
\varepsilon_0$ for some positive constant $\varepsilon_0$. It follows from \eqref{eq44} that there exists a uniform
constant $\delta_0>0$ such that
$$|\nabla\Psi(x)|\geq \varepsilon_0/2\quad\quad \text{for\ any}\quad
x\in B_{\delta_0}\left(x^{(i)}\right),$$ and
$$B_{\delta_0(x^{(i)})}\cap
B_{\delta_0(x^{(j)})}=\emptyset\quad\quad\text{for}\quad i\neq j.$$
Furthermore, it holds that \be
\iint_{\bigcup_iB_{\delta_0}(x^{(i)})}|\nabla\Psi|^2dx_1dx_2=\sum_{i=1}^{+\infty}\iint_{B_{\delta_0}(x^{(i)})}|\nabla\Psi|^2dx_1dx_2=\infty,
\ee which contradicts to (\ref{d3}).

\subsection{Positivity of horizontal velocity except at the corners}
It follows from Proposition \ref{prop21} that
the flow $(\rho, u, v) =(H, \frac{\partial_2\psi}{H}, -\frac{\partial_1\psi}{H})$ is indeed a solution of two dimensional steady Euler equations \eqref{a0} with boundary conditions \eqref{a01}-\eqref{upMach}, if it satisfies \eqref{eq21} and \eqref{eq22}. In order to justify \eqref{eq21} and \eqref{eq22}, now we need only to show that the horizontal velocity of the subsonic solution is always positive except at the corner points, namely
$\p_{x_2}\psi>0$ in $\bar\O\backslash\{P_1,P_2\}$.

Since
$$\p_{x_2}\psi=\p_{x_2}\Psi+\bar\psi'(x_2)=\p_{x_2}\Psi+\rho_0u_0(x_2)$$
and $|\nabla\Psi| \rightarrow 0$ as $|x|\rightarrow\infty$, there exist $R>0$ and $\delta>0$ such that $\p_{x_2}\psi>\delta$ for
$|x|>R$. Hence, the horizontal velocity is positive for $\{(x_1, x_2)||x_1|\geq R, \, x_2>R\}$.  Let $\O'=\left\{(x_1,x_2)\in \R^2|
f(x_1)<x_2<R, |x_1|< R\right\}$. Combining the argument \cite[Lemma2]{XX3} and the convexity condition in \eqref{a2} which implies
\[
F(\psi)F''(\psi)+\left(F'(\psi)\right)^2\geq 0,
\]
yields the positivity of the horizontal velocity for subsonic flows
in bounded domain $\Omega'$. Consequently, the horizontal velocity
of the subsonic solutions is positive in the whole domain $\O$. Thus $(\rho, u, v)$ is  a solution of the Euler system \eqref{a0} with boundary conditions \eqref{a01}-\eqref{upMach}.

\section{The uniqueness of subsonic Euler flows past a wall}\label{secunique}
In this section, we show the uniqueness of subsonic solutions
satisfying \eqref{d3} and  the asymptotic behavior
\eqref{a6}-\eqref{a16} in the far fields.

Assume that $\psi_{(i)}\in C^{2,\a}(\O)\cap C^{1,\b}(\bar\O)$ ($i=1,2$)
solve the problem \eqref{b01} and satisfy \be\label{e3}
\psi_{(i)}-\bar{\psi}\in L^{\infty}(\O),\quad
\nabla(\psi_{(i)}-\bar{\psi})\in L^2(\O),\quad \text{ and}\quad
|\nabla\psi_{(i)}|^2\leq(1-2\e_0)\Sigma^2(\B(\psi_{(i)}, \rho_0))  \ee for some
$\e_0>0$. Set $\phi=\psi_{(1)}-\psi_{(2)}$. Then it follows from \eqref{eq342} and \eqref{t1} that $\phi$ satisfies
\begin{equation}
\|\phi\|_{L^\infty(\Omega)}\leq C\rho_0\quad \text{and}\quad
\|\nabla\phi\|_{L^2(\Omega)}\leq \msC.
\end{equation}

Furthermore, $\phi$ solves the following problem
\be\label{e4}\left\{\ba{l}
\p_i\left(a_{ij}(\nabla \psi_{(1)}, \psi_{(1)};\nabla\psi_{(2)}, \psi_{(2)})\p_j\phi\right)+\p_i(b_i(\nabla \psi_{(1)}, \psi_{(1)};\nabla\psi_{(2)}, \psi_{(2)})\phi)\\
\quad =b_i(\nabla \psi_{(1)}, \psi_{(1)};\nabla\psi_{(2)}, \psi_{(2)})\p_i\phi+d(\nabla \psi_{(1)}, \psi_{(1)};\nabla\psi_{(2)}, \psi_{(2)})\phi\quad\quad
\text{in}\quad\O,\\
\phi=0\quad\quad \text{on}\quad x_2=f(x_1),
\ea\right. \ee where
$a_{ij}$, $b_i$, and $d$ are defined in \eqref{c25}-\eqref{c28} with $H_L$ replaced by $H$.
In this section, we denote  $H_{(t)}=H(|\nabla \psi_{(1+t)}|^2,\psi_{(1+t)}; \rho_0)$ with  $\psi_{(1+t)}=(1-t)\psi_{(1)}+t\psi_{(2)}$ for $t\in [0,1]$.

Denote  $B^\O_{r}(0)= B_{r}(0)\cap \O$ for $r > 0$ and let $\eta$ be
the smooth cut-off function satisfying \be\label{e9} \eta
=\left\{\ba{ll}
1\quad\quad&\text{for}\quad (x_1,x_2)\in B^\O_R(0),\\
0\quad\quad&\text{for}\quad (x_1,x_2)\in \O\backslash
B_{2R}(0),\ea\right. \ee and $|\g \eta| \leq 2/R$.

Multiplying the both sides of the equation in \eqref{e4} with
$\eta^2\phi$ and integrating in $\O$ yield that \be\label{e6}
\ba{rl}
&\iint_{B^\O_{2R}(0)}\eta^2 \left[\int^1_0\f{|\nabla
\phi|^2}{H_{(t)}}+\f{\left|\nabla\phi\cdot \nabla
\psi_{(1+t)}-F(\psi_{(1+t)})F'(\psi_{(1+t)})H_{(t)}^2\phi\right|^2}{H_{(t)}(\gamma H_{(t)}^{\gamma+1}-|\nabla
\psi_{(1+t)}|^2)}dt \right]dx_1dx_2\\
& + \iint_{B^\O_{2R}(0)}\eta^2 \left[\int^1_0 \phi^2\int^1_0H_{(t)}\left(F(\psi_{(1+t)})F''(\psi_{(1+t)})+(F'(\psi_{(1+t)}))^2\right)dt\right]dx_1dx_2\\
=&-2\iint_{B^\O_{2R}(0)}\left(a_{ij}\eta\p_i\eta\phi\p_j\phi-\eta\p_i\eta\phi\int^1_0\f{\p_i\psi_{(1+t)} F(\psi_{(1+t)})F'(\psi_{(1+t)})H_{(t)}^2}{(\gamma H_{(t)}^{\gamma+1}-|\nabla \psi_{(1+t)}|^2)}dt\right)dx_1dx_2.
\ea \ee
Note \be\label{e8} \int^1_0 H_{(t)}\left(F(\psi_{(1+t)})F''(\psi_{(1+t)})+(F'(\psi_{(1+t)}))^2\right)dt =\int_0^1
\f{u_0''\left(\kappa(\psi_{(1+t)};\rho_0)\right)}{\rho_0^2
u_0\left(\kappa(\psi_{(1+t)};\rho_0)\right)} H_{(t)}dt \geq 0. \ee
Therefore, it follows from \eqref{e6} that one has
\begin{equation}\label{e11}
\begin{aligned}
&\iint_{B^\O_{2R}(0)}\eta^2 \left[\int^1_0 \f{\left|\nabla\phi\cdot \nabla
\psi_{(1+t)}-F(\psi_{(1+t)})F'(\psi_{(1+t)})H_{(t)}^2\phi\right|^2}{H_{(t)}(H_{(t)}^2 c^2(H_{(t)})-|\nabla
\psi_{(1+t)}|^2)}dt \right]dx_1dx_2\\
\leq&\msC\iint_{B^\O_{2R}(0)}\eta \left( |\phi|\left|\nabla\eta\cdot \nabla\phi\right| + |\nabla\eta| \int_0^1 \left|F(\psi_{(1+t)})F'(\psi_{(1+t)}) \phi\right|dt \right)dx_1dx_2.
\end{aligned}
\end{equation}
Denote $A^\O_{R, 2R}(0)=B^\O_{2R}(0)\backslash \overline{B^\O_R(0)}$. Noting that $\phi$ and $\psi_{(1+t)}$ are uniformly bounded and
 using the Young inequality gives
\begin{equation}\label{5tech}
\begin{aligned}
& \iint_{B^\O_{2R}(0)}\eta^2 \left(\int_0^1\left|F(\psi_{(1+t)})F'(\psi_{(1+t)})\phi\right|^2dt\right) dx_1dx_2\\
\leq &   \iint_{B^\O_{2R}(0)}\eta^2 \left[\int^1_0\f{|\nabla\phi\cdot \nabla
\psi_{(1+t)}|^2}{H_{(t)}(\gamma H_{(t)}^{\gamma+1}-|\nabla\psi_{(1+t)}|^2)}dt \right]dx_1dx_2\\
&+\iint_{B^\O_{2R}(0)}\eta^2 \left[\int^1_0 \f{\left|\nabla\phi\cdot \nabla
\psi_{(1+t)}-F(\psi_{(1+t)})F'(\psi_{(1+t)})H_{(t)}^2\phi\right|^2}{H_{(t)}(\gamma H_{(t)}^{\gamma+1}-|\nabla
\psi_{(1+t)}|^2)}dt \right]dx_1dx_2\\
\leq&\msC\iint_{B^\O_{2R}(0)} |\nabla\phi|^2dx_1dx_2+ \msC\iint_{A^\O_{R, 2R}(0)}|\nabla\eta\cdot \nabla\phi| dx_1dx_2\\
&+ \msC\delta \iint_{B^\O_{2R}(0)}\eta^2\left(\int_0^1 \left|F(\psi_{(1+t)})F'(\psi_{(1+t)}) \phi\right|dt \right)^2 dx_1dx_2 + \msC(\delta)\iint_{A^\O_{R,2R}(0)}|\nabla\eta|^2 dx_1dx_2\\
\leq&\msC\iint_{B^\O_{2R}(0)} |\nabla\phi|^2dx_1dx_2 + \msC\delta \iint_{B^\O_{2R}(0)}\eta^2 \int_0^1 \left|F(\psi_{(1+t)})F'(\psi_{(1+t)}) \phi\right|^2dt  dx_1dx_2\\
& + \msC(\delta)\iint_{A^\O_{R, 2R}(0)}|\nabla\eta|^2 dx_1dx_2.
\end{aligned}
\end{equation}
Choosing a suitable small $\delta$ yields that
\be\label{e12}
\ba{rl}
&\iint_{B^\O_{R}(0)} \int_0^1\left|F(\psi_{(1+t)})F'(\psi_{(1+t)}) \phi\right|^2dt ~dx_1dx_2\\
\leq&\msC\iint_{B^\O_{2R}(0)} |\nabla\phi|^2dx_1dx_2  + \msC\iint_{A^\O_{R, 2R}(0)}\f1{R^2} dx_1dx_2\\
\leq& \msC.
\ea
\ee
It follows from \eqref{e6} that one has
\be\label{e7}
\ba{rl}
&\iint_{B^\O_R(0)}|\nabla\phi|^2dx_1dx_2\\
\leq& \msC\iint_{A^\O_{R, 2R}(0)}|\nabla\eta|\left(|\nabla\phi|+\phi\int^1_0F(\psi_{(1+t)})F'(\psi_{(1+t)})dt\right)dx_1dx_2\\
\leq& \msC\|\nabla\eta\|_{L^2\left(A^\O_{R,
2R}(0)\right)}\left(\|\nabla\phi\|_{L^2\left(A^\O_{R,
2R}(0)\right)}+\left\|\phi\int^1_0F(\psi_{(1+t)})F'(\psi_{(1+t)})dt\right\|_{L^2\left(A^\O_{R,
2R}(0)\right)}\right). \ea \ee
In view of \eqref{e12} and letting $R\rightarrow+\infty$ yields that
\[
\iint_{\O}|\nabla\phi|^2dx_1dx_2=0.
\]
Thus $\nabla\phi=0$ in $\O$. Since $\phi=0$ on $\p\O$, one has $\phi= 0$ in $\O$.
This proves the
uniqueness of subsonic solution to the problem \eqref{b01}.

\section{Existence of the critical density in the upstream}\label{seccritical}

In this section, we  show that there exists a critical density $\rho_{cr}$ such that there exists a subsonic solution as long as the density of the incoming flows is less than $\rho_{cr}$.

\begin{prop}\label{prop61}
There exists
 a critical value $\underline\rho_0> 0$, such that if
 $\rho_0>\underline\rho_0$, there exists a unique $\psi$ which solves
 the following problem
\be\label{d1} \left\{\ba{ll}
\Div\left(\f{\nabla\psi}{H(|\nabla\psi|^2,\psi;\rho_0)}\right)=F(\psi;\rho_0)F'(\psi;\rho_0)H\quad&\text{in}\quad\O,\\
\psi=0\ \ \
&\text{on}\ \ \ \Gamma,\ea\right.  \ee and satisfies
 \be\label{d30}
\psi\geq 0,\,\,|\psi-\bar\psi| \leq C\rho_0\ \ \text{in}\ \ \bar\O, \ \ \|\nabla(\psi-\bar)\psi\|_{L^2(\Omega)}\leq \msC\rho_0,\,\, \text{and}\ \ \
M(\rho_0)=\sup_{\bar\O}\f{|\nabla\psi|}{\Sigma(\B(\psi;\rho_0))}<1.
 \ee
 Moreover, either $M(\rho_0)\rightarrow 1$ as
 $\rho_0\downarrow\underline\rho_0$ or there does not exist a $\sigma>0$
 such that the problem \eqref{d1} has a solution for all $\rho_0\in(\underline\rho_0-\sigma,\underline\rho_0)$ and
 \be\label{d31}
\sup_{\rho_0\in(\underline\rho_0-\sigma,\underline\rho_0)}M(\rho_0)<1.
 \ee
\end{prop}
\begin{proof}
The key idea of the proof is similar to the proof of \cite[Proposition 6]{XX3}.

Let $\left\{\e_n\right\}_{n=1}^\infty$ be a strictly decreasing
positive sequence satisfying $\e_1\leq 1/4$ and $\e_n\rightarrow0$ as
$n\rightarrow\infty$, and $\zeta_n$ be a sequence of smooth
increasing odd functions satisfying
\be\label{d7}
\zeta_n(z)=\left\{\ba{ll}
z\ \ \ \ \ \ & \text{if}\ \ |z|<1-2\e_n,\\
1-3\e_n/2 &\text{if}\ \ z\geq 1-\e_n. \ea\right. \ee For
$\cmB_L(\psi; \rho_0)$ defined in \eqref{defcWB},  set
$${\check{H}}_L^{(n)}(|\nabla \psi|^2, \psi;
\rho_0)=\mclH\left(\zeta_n^2\left(\frac{|\nabla\psi|}{\Sigma(\cmB_L(\psi;
\rho_0))}\right)\Sigma^2(\cmB_L(\psi; \rho_0)),
\cmB_L(\psi;\rho_0)\right).$$ We first study the problem
\begin{equation}\label{Mbvpn}
\left\{
\begin{aligned}
& \left(\left(1-\zeta^2\left(\frac{|\nabla\psi|}{\sqrt{\gamma}(\cH^{(n)}_L)^{\frac{\gamma+1}{2}}}\right)\right)\delta_{ij} +\zeta\left(\frac{\partial_i\psi}{\sqrt{\gamma}(\cH^{(n)}_L)^{\frac{\gamma+1}{2}}}\right)\zeta\left(\frac{\partial_j\psi}{\sqrt{\gamma}(\cH^{(n)}_L)^{\frac{\gamma+1}{2}}}\right)
\right)\partial_{ij}\psi\\
& \quad \quad =\frac{\cW_L(\cH^{(n)}_L)^2}{1-\zeta^2\left(\frac{|\nabla\psi|}{\sqrt{\gamma}(\cH^{(n)}_L)^{\frac{\gamma+1}{2}}}\right)}\quad \text{in}\,\,\Omega_L,\\
& \psi=0\quad \text{on}\,\, \Gamma, \quad \psi=m_L \quad \text{on}\,\, \Gamma_L,
\end{aligned}
\right.
\end{equation}
where $m_L$ is defined in \eqref{defFmL}. Similar to the proof for
Lemma \ref{lemma31}, one can show that the problem \eqref{Mbvpn} has
a solution $\psi_L$ for any $\rho_0> \rho_0^*$. Given $\rho_0\in
(\rho_0^*, \infty)$, let $\mathcal{S}_L^{(n)}(\rho_0)$ be the set of
all solutions of the problem \eqref{Mbvpn}. Denote
\begin{equation}
\mathcal{M}_L^{(n)}(\rho_0) =\sup_{\psi_L\in \mathcal{S}_L(\rho_0)} \sup_{x\in \Omega_L}\frac{|\nabla\psi_L(x;\rho_0)|}{\Sigma(\B_L(\psi_L; \rho_0))}
\end{equation}
and
\begin{equation}
\begin{aligned}
\underline\rho_{0,L}^{(n)}= \inf\left\{s|  \text{for any}\,\, \rho_0\geq s, \mathcal{M}_L^{(n)}(\rho_0) \leq 1-3\e_n \right\}.
\end{aligned}
\end{equation}
It is easy to see that $\rho_{0,L}^{(n)} \leq \bar\rho_0$.
For any $\rho_0> \underline{\rho}_{0,L}^{(n)}$, Corollary \ref{cor35} and Lemma \ref{lem37} show that the associated subsonic solution $\psi_L$ of the problem \eqref{Mbvpn} satisfies
\begin{equation}
|\psi_L-\bar\psi_L|\leq C\rho_0\quad \text{and}\quad \|\psi_L-\bar\psi_L\|_{L^2(\Omega_L)}\leq \msC^{(n)}\rho_0,
\end{equation}
where $\msC^{(n)}$ depends on $\e_n$. Define
\begin{equation}
\underline{\rho}_0^{(n)} =\liminf_{L\to \infty} \underline{\rho}_{0,L}^{(n)}.
\end{equation}
If  $\rho_0>\underline{\rho}_0^{(n)}$, there exists a solution $\psi(\cdot, \rho_0)$ of the problem \eqref{b01} satisfying
\[
 \sup_{\Omega}\frac{|\nabla\psi(\cdot; \rho_0)|}{\sqrt{\gamma} H^{\frac{\gamma+1}{2}}(|\nabla \psi|^2, \psi; \rho_0)}\leq 1-4\e_n.
\]
Furthermore, it is easy to see that $\{\underline{\rho}_0^{(n)}\}$ is a decreasing sequence.
Define $\underline \rho_0=\inf\underline{\rho}_0^{(n)}$. If $\rho_0>\underline\rho_0$, there is always a solution $\psi$ of the problem \eqref{d1}.
 The same
argument in \cite[Proposition 6]{XX3} gives that $\underline\rho_0$ is the critical value described in Proposition \ref{prop61}. This finishes
the proof of the proposition.
\end{proof}

As a direct consequence of Proposition 6.1,  we complete
the proof of Theorem  \ref{th1}.

\section{Subsonic Euler flows with general incoming velocity}\label{secgeneral}
In this section, we  consider the existence and uniqueness of subsonic Euler flows past a wall when the incoming horizontal velocities are general small perturbations of a constant.

{\bf Proof of Theorem \ref{th2}.} Theorem \ref{th2} is proved in a
similar fashion as that for Theorem \ref{th1}, so we only sketch the proof and emphasize on
the main differences as follows. The main difference is that the convex
condition $u''_0(x_2)\geq 0$ in Theorem \ref{th1} is replaced by the
smallness of $u_0'(x_2)$ in \eqref{a05}.

{\bf Step 1. Subsonic solutions in nozzles and their uniform estimates.} For $L>0$ sufficiently large,
let $g_L$ be defined in \eqref{b4}
and $u_{0,L}(x_2)=u_0(0)+\int_0^{x_2}g_L(s)ds$. Hence for $x_2\in [0, L]$, we have
\[
u_{0,L}(x_2) \geq \inf_{x_2\in [0, L-1]}u_{0}(x_2) -\frac{|u_0'(L-1)|}{2}.
\]
If $L$ is sufficiently large, then $u_{0,L}(x_2) \geq \bar u/2$ for all $x_2\in [0, L]$.
Furthermore,  $u_{0,L}(x_2)$ satisfies $u_{0,L}'(L)=g_L(L)=0$.

Let $F_L(\psi; \rho_0)$, $W_L(\psi;\rho_0)$, $m_L$ be the same as that in \eqref{defFmL}.
We have the following proposition.
\begin{prop}
For any $k>1$, there exists a constant $\e_1>0$ independent of $L$ such that if $u_0(x_2)$ satisfies the conditions
\eqref{a04}-\eqref{a05} with $\e\in (0, \e_1)$, then there exists a $\bar
\rho_0\in (\rho_0^*, \infty)$ independent of $L$, such that if $\rho_0>\bar\rho_0$, there exists
a solution $\psi_L\in C^{2,\a}(\O_L)\cap C^{1,\b}(\bar\O_L)$ of the problem \eqref{Tbvp} and satisfies \be\label{g4} 0\leq\psi_L\leq m_L \
\ \text{in}\ \ \O_L \quad \text{and}\quad
\sup_{x\in\Omega_L}\frac{|\nabla\psi_L|}{\Sigma(\B_L(\psi_L;\rho_0))}\leq \frac{1}{4}. \ee
\end{prop}
\begin{proof}
First, let $\check{F}_L$ be the same as that in \eqref{defFL}.
Define
\begin{equation}
\check{\BB}_L(\Psi, x;\rho_0) =h(\rho_0) +\frac{\check{F}_L^2(\bar\psi_L +\bar S\zeta(\frac{\Psi}{\bar S});\rho_0)}{2}
\end{equation}
and
\begin{equation}
\check{\mfW}_L(\Psi, x;\rho_0) = \check{F}_L\left(\bar\psi_L +\bar S\zeta\left(\frac{\Psi}{\bar S}\right);\rho_0\right)\check{F}_L'\left(\bar\psi_L +\bar S\zeta\left(\frac{\Psi}{\bar S}\right);\rho_0\right)
\end{equation}
where $\zeta$ and $\bar\psi_L$ are defined in \eqref{defzeta} and \eqref{defhpsiL}, respectively,  and
\[
\bar S =2\rho_0 (\bar U+1)\quad \text{and}\quad \bar U = \sup_{x_1\in\R}\int^{f(x_1)}_0u_{0}(s)ds.
\]
If $x_2\geq \frac{2\bar S}{\rho_0 \min u_{0,L}}$, then
\begin{equation}
\frac{\rho_0 \min u_{0,L}}{2} x_2 \leq \bar\psi_L  +\bar S\zeta\left(\frac{\Psi}{\bar S}\right)\leq 2\rho_0\max u_{0,L} x_2.
\end{equation}
If $x_2\in \left(0,  \frac{2\bar S}{\rho_0 \min u_{0,L}}\right)$, then
\begin{equation}
 \bar\psi_L  +\bar S\zeta\left(\frac{\Psi}{\bar S}\right)\geq -\bar S.
\end{equation}
Therefore, it follows from \eqref{a05} that there exists a uniform constant $C$ such that
\begin{equation}\label{estmfW}
|\check{\mfW}_L(\Psi, x;\rho_0)|\leq \frac{C\e}{\rho_0(1+x_2)^{k+1}}.
\end{equation}
Set $\cfH_L(\nabla \Psi, \Psi, x;\rho_0)=\mclH
\left(\zeta^2\left(\frac{|\nabla (\bar\psi_L
+\Psi)|}{\Sigma(\check{\BB}_L(\Psi, x;\rho_0))}\right)
\Sigma^2(\check{\BB}_L(\Psi, x;\rho_0)),
\check{\BB}_L(\Psi;\rho_0)\right)$ and
$\Omega_{L,N}=\Omega_L\cap\{x\big| |x_1|\leq N\}$. We first study
the problem
\begin{equation}\label{sup4}
\left\{
\begin{aligned}
&\mfA_{L;ij}(\nabla \Psi, \Psi, x;\rho_0) \partial_{ij}\Psi = Q_L(\nabla \Psi, \Psi, x; \rho_0)\quad \text{in}\,\, \Omega_{L,N},\\
& \Psi=-\rho_0\int_0^{f(x_1)}u_{0,L}(s)ds\quad\text{on}\quad\Gamma\cap \partial\Omega_{L, N},\quad \Psi=0\quad\text{on}\quad\partial \Omega_{L,N}\setminus\Gamma,
\end{aligned}
\right.
\end{equation}
where
\begin{equation}
\begin{aligned}
\mfA_{L;ij}(\nabla \Psi, \Psi, x;\rho_0) =&\left( 1-\zeta^2\left(\frac{|\nabla(\Psi_L+\bar\psi_L)|}{\sqrt{\gamma}\cfH_L^{\frac{\gamma+1}{2}}(\nabla \Psi, \Psi, x;\rho_0)}\right)\right)\delta_{ij} \\
&+\zeta\left(\frac{\partial_i(\Psi_L+\bar\psi_L)}{\sqrt{\gamma}\cfH_L^{\frac{\gamma+1}{2}}(\nabla \Psi, \Psi, x;\rho_0)}\right)
\zeta\left(\frac{\partial_j(\Psi_L+\bar\psi_L)}{\sqrt{\gamma}\cfH_L^{\frac{\gamma+1}{2}}(\nabla \Psi, \Psi, x;\rho_0)}\right)
\end{aligned}
\end{equation}
and
\[
 Q_L(\nabla \Psi, \Psi, x; \rho_0)=\check{\mfW}_L(\Psi, x;\rho_0)\left(\cfH_L^2(\nabla \Psi, \Psi, x;\rho_0)-\rho_0^2 -\rho_0^2 \zeta^2\left(\frac{|\partial_1\Psi|}{\sqrt{\gamma}\cfH_L^{\frac{\gamma+1}{2}}(\nabla \Psi, \Psi, x;\rho_0)}\right)\right).
\]
It follows from \eqref{estmfW} that $Q_L$ satisfies
\be
|Q_L(\nabla \Psi, \Psi, x; \rho_0)|\leq  \f{C_1 \rho_0 \varepsilon}{(1+x_2)^{k+1}},
\ee
where $C_1$ is independent of $L$. The eigenvalues for the matrix $A_L$ are
\begin{equation}
\lambda_L = 1-\sum_{i=1}^2\zeta^2\left(\frac{|\partial_i(\Psi_L+\bar\psi_L)|}{\sqrt{\gamma}\cfH_L^{\frac{\gamma+1}{2}}(\nabla \Psi, \Psi, x;\rho_0)}\right)
\end{equation}
and
\begin{equation}
\Lambda_L = 1-\sum_{i=1}^2\zeta^2\left(\frac{|\partial_i(\Psi_L+\bar\psi_L)|}{\sqrt{\gamma}\cfH_L^{\frac{\gamma+1}{2}}(\nabla \Psi, \Psi, x;\rho_0)}\right) + \zeta^2\left(\frac{|\nabla(\Psi_L+\bar\psi_L)|}{\sqrt{\gamma}\cfH_L^{\frac{\gamma+1}{2}}(\nabla \Psi, \Psi, x;\rho_0)}\right).
\end{equation}
As usual,  $\Lambda_L/\lambda_L$ is called the elliptic ratio for the equation in \eqref{sup4}.

Let $\hat{\phi}=\f{\e\rho_0}{(1+x_2)^{k-1}}$. Obviously, one has
$\partial_{x_1}(\bar\psi+C\hat\phi)=0$ for any constant $C$.  Hence
direct computations give
\[
\mfA_{L; ij}(\nabla(C_2\hat\phi), \Psi, x; \rho_0)\p_{ij}(C_2\hat{\phi})=C_2 \partial_{x_2x_2} \left(\f{\e\rho_0}{(1+x_2)^{k-1}}\right)= \frac{C_2 k(k-1)\rho_0\e}{(1+x_2)^{k+1}}.
\]
Therefore, choosing $C_2$ sufficiently large yields
that
$$\left\{
\ba{ll}
\mfA_{L;ij}(\nabla(C_2\hat\phi), \Psi, x; \rho_0)\p_{ij}(C_2\hat{\phi})\geq Q_L= \mfA_{L;ij}(\nabla \Psi, \Psi, x;\rho_0)\partial_{ij}\Psi_{L,N}\quad\quad&\text{in}\quad \O_{L,N},\\
C_2\hat\phi \geq 0\geq -\rho_0\int_0^{f(x_1)}u_{0,L}(s)ds=\Psi_{L,N}\quad&\text{on}\quad \Gamma\cap \p\O_{L,N},\\
C_2\hat\phi \geq 0=\Psi_{L,N}&\text{on}\quad\partial\Omega_{L,N}\setminus\Gamma.
\ea\right.
$$
Therefore, the comparison principle for nonlinear elliptic equations \cite[Theorem 10.1]{GT} implies that
$$\Psi_{L,N}\leq C_2\hat \phi\leq C_2\frac{\rho_0 \varepsilon}{(1+x_2)^{k-1}}\quad\quad\text{in}\quad \O_L.$$
One of our key observation is that the constant $C_2$ depends
neither on $L$ nor on the elliptic coefficients. Note also that
$C_2$ does not depend on the elliptic ratio $\Lambda_L/\lambda_L$.

Similarly, one has $$\Psi_{L,N}\geq -\rho_0\bar U -C_2\frac{\rho_0 \varepsilon}{(1+x_2)^{k-1}}.$$ Therefore, choosing $\e\leq \f1{C_2}$ yields that for sufficiently large $L$, the following estimate holds
$$-\bar{U}-\rho_0<\Psi_{L,N}<\rho_0\quad\quad\text{in}\quad \O_L.$$
Thus
we have the following uniform $L^{\infty}$-norm estimate
\be\label{sc15}
|\Psi_{L,N}|\leq \rho_0\left(\bar U+1\right).
\ee

Next, similar to the proof for Lemma \ref{lem36} in Subsection 3.2,
it follows from the H\"{o}lder gradient estimate \cite[Theorem
12.4]{GT} and the estimate near the corners \cite[Section 2.2]{DX}
for the elliptic equations that  one has  the following global
estimate \be\label{sc9} \|\Psi_{L, 2N}\|_{1,\beta;\O_{L, N}}\leq
\msC \rho_0, \ee where $\msC$ depends on the elliptic coefficients
but is independent of $L$. Taking limit for $N\to \infty$, one gets
that there exists a subsequence of $\{\Psi_{L,N}\}$ converging to
$\Psi_L$ which satisfies \be\label{7uestPsiL} |\Psi_{L}|\leq
\rho_0\left(\bar U+1\right)\quad \text{and}\quad
\|\Psi_{L}\|_{1,\beta;\O_{L}}\leq \msC \rho_0. \ee If $\rho_0$ is
sufficiently large, then
\begin{equation}
\frac{|\nabla (\Psi_L+\bar \psi_L)|}{\Sigma(\check{\BB}_L)}  \leq \frac{|\nabla \Psi_L|+|\nabla \bar\psi_L|}{\Sigma(\check{\BB}_L)}\leq \msC^{\sharp}\rho_0^{\frac{1-\gamma}{2}}\leq \frac{1}{4}.
\end{equation}
Therefore, $\psi_L=\bar\psi_L+\Psi_L$ solves the problem \eqref{Tbvp}. Since $\cW_L\leq 0$ if $\psi_L\geq m_L$ and $\cW_L\geq 0$ if $\psi_L\leq 0$,  it follows from the maximum principle that
\begin{equation}
0\leq \psi_L\leq m_L.
\end{equation}
We now choose $\e_1=\frac{1}{C_2}$, then the proposition is proved.
\end{proof}

Furthermore, we also have the following uniform integral estimate.
\begin{lemma}
For any $k>1$, there exists a constant $\e_0>0$ independent of $L$ such that if $u_0(x_2)$ satisfies the conditions
\eqref{a04}-\eqref{a05} with $\e\in (0, \e_0)$,
and $\psi_L$ is a subsonic solution of the problem \eqref{Tbvp}, then we have
\begin{equation}\label{7psiLfar}
|(\psi_L-\bar \psi_L)(x_1, x_2)|\to 0\quad \text{uniformly with respect to } x_2\text{ as } |x_1|\to\infty
\end{equation}
and
 \be\label{7L2integral}
\|\nabla(\psi_L-\bar\psi_L)\|_{L^2(\O_L)}\leq \msC, \ee where the constant $\msC$ depends on the elliptic coefficients and is independent of $L$.
\end{lemma}
\begin{proof}
The far field behavior \eqref{7psiLfar} follows from  \cite[Proposition 4]{XX3}.

In fact, it follows from \eqref{L2-0} that \be\label{sL2-0}
\ba{lll}
 &\iint_{\O_{L,N}}\left(a_{ij}\partial_i\Psi_L\partial_j\Psi_L+2b_i\Psi_L\partial_i\Psi_L+d_1\Psi^2_L\right)dx_1dx_2\\
\leq& \msC(\rho_0)\left(\int_{x_1\in\mathbb{R}}|\bar{\psi}_L(f(x_1))|dx_1+\int_{f(N)}^{L}|\Psi_L(N,x_2)|dx_2+\int_{f(-N)}^{L}|\Psi_L(-N,x_2)|dx_2\right)\\
&-\iint_{\O_{L,N}}d_2\Psi^2_Ldx_1dx_2. \ea \ee  Although the sign of
the integral term $d_2$ is not clear without  the assumption
$u_{0,L}''(x_2)\geq 0$,  the following estimate holds
 \be\label{small}
 \begin{aligned}
\left|d_2\left(\nabla\psi_{L},\psi_{L}; \nabla\bar\psi_{L},\bar\psi_{L}\right)\right|= &\left|\int_0^1\f{H(|\nabla\bar\psi_{L,t}|^2, \bar\psi_{L,t})
u''_{0,L}\left(\kappa(\bar\psi_{L,t};\rho_0)\right)}{\rho_0^2u_{0,L}\left(\kappa_L(\bar\psi_{L,t};\rho_0)\right)}dt\right|\\
\leq &
\f{C\e}{\rho_0\left(1+\kappa_L(\bar\psi_{L,t};\rho_0)\right)^{k+2}}.
\end{aligned}
\ee
Since
$$|\Psi_L|=|\psi_L-\bar{\psi}_L|=\left|\int^{\kappa_L(\psi_L;\rho_0)-x_2}_0\rho_0u_{0,L}(s)ds\right|=\rho_0 \max u_{0,L}|\kappa_L(\psi_L;\rho_0)-x_2|,$$
it follows from \eqref{7uestPsiL} that
$$|\kappa_L(\psi_L;\rho_0)-x_2|\leq \frac{\bar U+1}{\max u_0}.$$
Therefore,
\begin{equation}
-\frac{\bar U+1}{\max u_0}+x_2\leq \kappa(\bar\psi_{L,t};\rho_0)\leq \frac{\bar U+1}{\max u_0}+x_2.
\end{equation}
Thus
\begin{equation}\label{estud2}
|u_{0,L}''(\kappa(\bar\psi_{L,t};\rho_0))|\leq \frac{C\e}{(1+x_2)^{k+2}}.
\end{equation}
This, together with \eqref{small} implies that \be\label{small2}
\left|d_2\left(\nabla\psi_{L},\psi_{L}; \nabla\bar\psi_{L},\bar\psi_{L}\right)\right|\leq
\f{C \e}{\rho_0(1+x_2)^{k+2}}. \ee Therefore, combining
\eqref{small2}, \eqref{L2-0}, and \eqref{c14} together yields that for
any $N>0$, one has
$$\ba{ll}&\iint_{\O_{L,N}}|\nabla\Psi_L|^2dx_1dx_2\\
\leq&
\msC(\rho_0)\left(\int_{\mathbb{R}}|\bar{\Psi}_L(f(x_1))|dx_1+\int_{f(N)}^{L}|\Psi_L(N,x_2)|dx_2+\int_{f(-N)}^{L}|\Psi_L(-N,x_2)|dx_2\right)\\
&+C\iint_{\O_{L,N}}\f{\varepsilon\Psi^2_L}{(1+x_2)^{k+2}}dx_1dx_2.\ea$$
Noting that $k>1$ so that we apply the weighted Poincar\'{e} inequality in Lemma \ref{lemA1} in Appendix for the last term in \eqref{sL}  to get
$$
\ba{ll}
&\iint_{\O_{L, N}}|\nabla\Psi_L|^2dx_1dx_2\\
\leq&
\msC(\rho_0)\left(\int_{\mathbb{R}}|\bar{\psi}_L(x_1, f(x_1))|dx_1+\int_{f(N)}^{L}|\Psi_L(N,x_2)|dx_2+\int_{f(-N)}^{L}|\Psi_L(-N,x_2)|dx_2\right) \\ &+\f{C}{k+1}\int_\R|\Psi_L(x_1,f(x_1))|^2dx_1 +\f{4C_3\e}{(k+1)^2}\iint_{\O_{L,N}}\left|\nabla\Psi_L\right|^2 dx_1dx_2.\ea
$$
Choosing $\e_0 \leq \min\left(\e_1, \f{(k+1)^2}{8C_3}\right)$ yields
\begin{equation}
\begin{aligned}
&\iint_{\O_{L}}|\nabla\Psi_L|^2dx_1dx_2 \\
\leq&
\msC(\rho_0)\left(\int_{\mathbb{R}}|\bar{\psi}_L(x_1, f(x_1))|dx_1+\int_{f(N)}^{L}|\Psi_L(N,x_2)|dx_2+\int_{f(-N)}^{L}|\Psi_L(-N,x_2)|dx_2\right)\\ &+\f2{k+1}\int_\R|\Psi_L(x_1,f(x_1))|^2dx_1.
\end{aligned}
\end{equation}
 Taking limit $N\rightarrow\infty$ and using \eqref{7psiLfar} show
\be\label{sL}
\int_{\O_{L}}|\nabla\Psi_L|^2dx_1dx_2\leq
\msC(\rho_0)\int_{\mathbb{R}}|\bar{\psi}_L(f(x_1))|dx_1 +\f{C}{k+1}\int_\R|\Psi_L(x_1,f(x_1))|^2dx_1. \ee
This finishes the proof of the lemma.
\end{proof}

{\bf Step 2. Existence of subsonic solutions and their fine properties. } The subsonic solution $\Psi$ on $\Omega$ is then obtained as a limit of $\{\Psi_L\}$.

\begin{prop}
Suppose that $u_0(x_2)$ satisfies the conditions
\eqref{a04}-\eqref{a05}, then there exists a $\bar
\rho_0\geq\rho_0^*$, such that if $\rho_0>\bar\rho_0$, there exists
a solution $\Psi\in C^{2,\a}(\O)\cap C^{1,\b}(\bar\O)$  satisfying
\be\label{sc14}
\Psi+\bar\psi\geq 0 \quad \text{and}\quad
\left|\nabla(\Psi+\bar\psi)\right|<\frac{1}{4}\Sigma(\Psi+\bar\psi;\rho_0).
\ee
\end{prop}
\begin{proof}
Taking the limit $L\rightarrow\infty$ gives
\begin{equation}
|\nabla (\bar\psi+\Psi)|\leq C\rho_0.
\end{equation}
Hence, if $\rho_0$ is sufficiently large, one has
\begin{equation}
\frac{|\nabla(\bar\psi+\Psi)|}{\Sigma(\B(\bar\psi+\Psi))}\leq 1/4.
\end{equation}
This shows that the subsonic truncation can be removed. So the proof of the proposition is finished.
\end{proof}

Furthermore, it follows from the estimate \eqref{7L2integral} that
\be\label{sL2-3} \|\nabla\Psi\|_{L^2(\mathbb{R}^2_+)}\leq C. \ee
Combining \eqref{sL2-3} with the H\"{o}lder gradient estimate yields
the asymptotic behavior of the flows. Using the same idea in
Subsection 4.2, one can prove that the flows in $\Omega$ have
positive horizontal velocity except at the corner points.

{\bf Step 3. The uniqueness of solution.} In order to prove the
uniqueness, we also study the problem for the difference of two
solutions. The same arguments in Section 5 give \eqref{e6}. Now the
key task is to estimate the following term
$$\iint_{\O}\eta^2\int^1_0H(|\nabla\psi_{(1+t)}|^2, \psi_{(1+t)})\phi^2\left(F(\psi_{(1+t)})F''(\psi_{(1+t)})+(F'(\psi_{(1+t)}))^2\right)dtdx_1dx_2,$$
where
$\eta$ is the smooth cut-off function with \eqref{e9}, $\phi$ is the difference of the two subsonic solutions $\psi_1$ and $\psi_2$, and $\psi_{(1+t)}=(1-t)\psi_1+t\psi_2$ for $t\in[0,1]$.

Note that
$$\left|F(\psi_{(1+t)})F''(\psi_{(1+t)})+\left(F'(\psi_{(1+t)})\right)^2\right|
=\left|\f{u''_0\left(\kappa(\psi_{(1+t)};\rho_0)\right)}{\rho^2_0u_0\left(\kappa(\psi_{(1+t)};\rho_0)\right)}\right|.
$$
The same argument as for \eqref{estud2} gives
\begin{equation}
u''_0\left(\kappa(\psi_{(1+t)};\rho_0)\right)\leq \f{C\e}{\rho_0^2(1+x_2)^{k+2}},
\end{equation}
where $C$ does not depend on $L$ or elliptic coefficients.
Therefore, \be\label{e012} \ba{rl}
& \iint_{B^\O_{2R}(0)}\eta^2\int^1_0H(|\nabla\psi_{(1+t)}|^2, \psi_{(1+t)})\phi^2\left(F(\psi_{(1+t)})F''(\psi_{(1+t)})+\left(F'(\psi_{(1+t)})\right)^2\right)dt dx_1dx_2\\
 \leq & C\e\iint_{\O} \f{\eta^2\phi^2}{(1+x_2)^{k+2}}dx_1dx_2.
\ea
\ee
Noting that $\phi=0$ on $\Gamma$ and
using the weighted Poincar\'e inequality in Lemma \ref{lemA1} yield
\be\label{7e012}
\ba{rl}
 & \iint_{B^\O_{2R}(0)}\eta^2\int^1_0H(|\nabla\psi_{(1+t)}|^2, \psi_{(1+t)})\phi^2\left(F(\psi_{(1+t)})F''(\psi_{(1+t)})+\left(F'(\psi_{(1+t)})\right)^2\right)dt dx_1dx_2\\
\leq &C\e\iint_{\O} \eta^2 |\nabla\phi|^2 + |\nabla\eta|^2 \phi^2 dx_1dx_2\\
\leq &C\e\iint_{B_{2R}^\O(0)} \eta^2|\nabla\phi|^2 dx_1dx_2 +
C\e\iint_{A^\O_{R,2R}(0)} |\nabla\eta|^2 dx_1dx_2. \ea \ee
Hence, substituting \eqref{7e012} into \eqref{e6} and applying the
same technique in  \eqref{e11} and \eqref{5tech} give that
\be\label{se6} \ba{rl}
&\iint_{B^\O_{2R}(0)} \int_0^1 \left|F(\psi_{(1+t)})F'(\psi_{(1+t)})\eta\phi\right|^2 dt dx_1dx_2\\
\leq &\msC\iint_{B_{2R}^\O(0)}|\nabla\phi|^2 dx_1dx_2 + \msC\delta\iint_{B_{2R}^\O(0)}\eta^2\left(\int_0^1F(\psi_{(1+t)})F'(\psi_{(1+t)})\phi dt\right)^2dx_1dx_2\\
&+\msC(\d)\iint_{A_{R,2R}^\O(0)}|\nabla\eta|^2dx_1dx_2,\ea \ee Choosing
suitable small $\delta > 0$ yields \be\label{se6} \ba{rl}
&\iint_{B^\O_{R}(0)} \int_0^1 \left|F(\psi_{(1+t)})F'(\psi_{(1+t)})\phi\right|^2dsdx_1dx_2\\
\leq &\msC\iint_{B_{2R}^\O(0)}|\nabla\phi|^2 dx_1dx_2
+\msC(\d)\iint_{A_{R,2R}^\O(0)}|\nabla\eta|^2dx_1dx_2.\ea \ee
Hence we obtain the key uniform estimate
$$
\left\|\phi\int_0^1\left|F(\psi_{(1+t)})F'(\psi_{(1+t)})\right| dt\right\|_{L^2(\O)}\leq \msC.
$$
This is exactly the same as \eqref{e12}. Similarly, one can prove \eqref{e7} which implies $\phi=0$ in $\Omega$. Hence the uniqueness for subsonic solution is proved.

{\bf Step 4. Existence of critical value $\rho_{cr}$ for the density  in the upstream.} Note that the choice of $\e_1$ and $\e_0$ does not depend on the elliptic coefficients, so the proof for the existence of critical value for the incoming density in the upstream  is similar to  the one in Section 6.
\hfill$\Box$

\section{Limit of subsonic flows}\label{seclimit}
In this section, we prove Theorem \ref{thmlimit}. Given a sequence of $\{M_{0}^{(n)}\}$ converging to $\rho_{cr}$, the assoicated subsonic flows $\{(\rho_n, u_n, v_n)\}$ satisfy the Euler system \eqref{a0} and the following three conditions
\begin{enumerate}
\item the Mach number of the flows $\frac{\sqrt{u_n^2+v_n^2}}{\sqrt{\gamma \rho_n^{\gamma-1}}}\leq 1$ a.e. in $\Omega$;
\item the Bernoulli functions of the flows $\frac{u_n^2+v_n^2}{2}+h(\rho_n)$ are uniformly bounded above and below;
\item the vorticities of the flows $\partial_{x_2}u_n -\partial_{x_1}v_n$ are uniformly bounded measures.
\end{enumerate}
Hence, using Theorem 2.2 in \cite{CHW} yields that there exists a subsequence still labelled by $\{(\rho_n, u_n, v_n)\}$ converging to $(\rho, u, v)$ a.e. in $\Omega$. Thus $(\rho, u, v)$ also solves the Euler system \eqref{a0} in the weak sense and the boundary condition \eqref{a01} in the sense of normal trace. This finishes the proof of Theorem \ref{thmlimit}.


\appendix

\section{The weighted Poincar\'{e} inequality}\label{secappend}
In this appendix, we give a weighted Poincar\'{e} inequality and its proof, which is used in Section 7.
\begin{lemma}\label{lemA1}
Let $I=(a, b)$ with $a\geq 0$ and $b$ could be infinity. If
$g=g(s)\in H^1(I)$, then for any $l>2$, it holds that \be\label{po}
\int_I\f{g^2(s)}{(s+1)^l}ds\leq
\frac{2g^2(a)}{l-1}+\f{4}{(l-1)^2}\int_I (g'(s))^2ds. \ee
\end{lemma}
\begin{proof} Integration by parts and using the Cauchy-Schwartz inequality give
\begin{equation*}
\begin{array}{ll}
&\int_I\f{g^2(s)}{(1+s)^l}ds\\
=&\f{1}{1-l}\left[\left.\f{g^2(s)}{(1+s)^{l-1}}\right|^{s=b}_{s=a}-2\int_Ig(s)g'(s)(1+s)^{1-l}ds\right]\\
\leq &\f{1}{l-1}\left[\f{g^2(a)}{(1+a)^{l-1}}-\f{g^2(b)}{(1+b)^{l-1}}\right]+\f{1}{l-1}\left[c\int_I\f{g^2(s)}{(1+s)^l}ds+\f{1}{c}\int_I\f{(g'(s))^2}{(1+s)^{l-2}}ds\right].
\end{array}
\end{equation*}
Hence,
$$(l-1-c)\int_I\f{g^2(s)}{(1+s)^l}ds\leq
g^2(a)+\f{1}{c}\int_I\f{(g'(s))^2}{(1+s)^{l-2}}ds.$$ Taking
$c=\f{l-1}{2}$ and noting that $l>2$ yield
\begin{equation*}
\int_I\f{g^2(s)}{(1+s)^l}ds\leq
\f{2}{l-1}g^2(a)+\f{4}{(l-1)^2}\int_I(g'(s))^2ds.
\end{equation*}
This finishes the proof of the inequality.
\end{proof}


{\bf Acknowledgments.}  Part of this work was done when the authors visited the Institute of Mathematical Sciences in The Chinese University of Hong Kong, Shanghai Jiao Tong University,  Sichuan University. They thank all these institutions for their hospitality and support. The work was completed when Chunjing Xie was a K C Wong visiting fellow at The Institut des Hautes \'{E}tudes Scientifiques (IHES) in France. He thanks IHES and K C Wong Education Foundation for their support.

\bibliographystyle{plain}

\end{document}